\documentclass[a4paper,12pt,reqno]{amsart}
\usepackage[english]{babel}
\usepackage[T1]{fontenc}
\usepackage[left=1in,right=1in,top=1.2in,bottom=1.2in]{geometry}
\usepackage{times}
\usepackage{microtype}

\usepackage{commath,amssymb,amscd,mathrsfs,mathtools,dsfont,bbm,multirow,array,caption,comment,pifont}
\usepackage[all]{xy}
\usepackage{enumerate}
\usepackage{longtable}

\usepackage{cite}			%won't show labels for citation

\allowdisplaybreaks

\usepackage[usenames,dvipsnames,svgnames,table]{xcolor}
\definecolor{red}{RGB}{255,25,25}
\definecolor{blue}{RGB}{25,50,200}
\usepackage{tikz-cd}
\usepackage[pagebackref,linktocpage]{hyperref}

\hypersetup{
pdftitle={},				% title
pdfauthor={Hu--Truong},	% author
colorlinks=true,			% false: boxed links; true: colored links
linkcolor=red,			% color of internal links (change box color with linkbordercolor)
citecolor=MidnightBlue,	% color of links to bibliography
filecolor=magenta,		% color of file links
urlcolor=MidnightBlue	% color of external links
}

\usepackage{cleveref}	% after hyperref!!

\newtheorem{theorem}{Theorem}[section]
\crefname{theorem}{Theorem}{Theorems}
\newtheorem{lemma}[theorem]{Lemma}
\crefname{lemma}{Lemma}{Lemmas}
\newtheorem{proposition}[theorem]{Proposition}
\crefname{proposition}{Proposition}{Propositions}

\crefname{prop}{Proposition}{Propositions}

\crefname{corollary}{Corollary}{Corollaries}

\crefname{cor}{Corollary}{Corollaries}
\newtheorem{conjecture}[theorem]{Conjecture}
\crefname{conjecture}{Conjecture}{Conjectures}

\crefname{conj}{Conjecture}{Conjectures}
\newtheorem*{conj*}{Conjecture}
\crefname{conj}{Conjecture}{Conjectures}

\crefname{conjA}{Conjecture}{Conjecture}

\crefname{conjB}{Conjecture}{Conjecture}

\newtheorem{conjC}{Conjecture}
\crefname{conjC}{Conjecture}{Conjecture}

\newtheorem{conjDk}{Conjecture}
\crefname{conjDk}{Conjecture}{Conjecture}

\newtheorem{conjD}{Conjecture}
\crefname{conjD}{Conjecture}{Conjecture}

\crefname{conjH}{Conjecture}{Conjecture}

\newtheorem{conjGr}{Conjecture}
\crefname{conjGr}{Conjecture}{Conjecture}

\theoremstyle{definition}
\newtheorem{definition}[theorem]{Definition}
\crefname{definition}{Definition}{Definitions}

\crefname{defn}{Definition}{Definitions}
\newtheorem{example}[theorem]{Example}
\crefname{example}{Example}{Examples}

\crefname{notation}{Notation}{Notation}
\newtheorem*{notation*}{Notation}
\crefname{notation}{Notation}{Notation}

\crefname{problem}{Problem}{Problems}

\crefname{question}{Question}{Questions}

\crefname{condition}{Condition}{Conditions}

\crefname{assumption}{Assumption}{Assumptions}

\crefname{propGr}{Property}{Property}

\theoremstyle{remark}

\crefname{rmk}{Remark}{Remarks}
\newtheorem*{rmk*}{Remark}
\crefname{rmk}{Remark}{Remarks}
\newtheorem{remark}[theorem]{Remark}
\crefname{remark}{Remark}{Remarks}

\crefname{fact}{Fact}{Facts}
\newtheorem{claim}[theorem]{Claim}
\crefname{claim}{Claim}{Claims}
\newtheorem*{claim*}{Claim}
\crefname{claim}{Claim}{Claims}

\crefname{step}{Step}{Steps}

\crefname{case}{Case}{Cases}

\numberwithin{equation}{section}
\newcommand{\wtilde}[1]{\widetilde{#1}}

\newcommand{\ol}[1]{\overline{#1}}

\newcommand{\dyncomp}{\,{\scriptstyle \Diamond}\,}
\newcommand{\dynpb}{\text{\ding{72}}}
\newcommand{\ratmap}{\dashrightarrow}
\newcommand{\injmap}{\hookrightarrow}

\newcommand{\lra}{\longrightarrow}

\newcommand{\bbP}{\mathbb{P}}

\newcommand{\bC}{\mathbf{C}}

\newcommand{\bF}{\mathbf{F}}

\newcommand{\bN}{\mathbf{N}}

\newcommand{\bQ}{\mathbf{Q}}
\newcommand{\bR}{\mathbf{R}}

\newcommand{\bZ}{\mathbf{Z}}

\newcommand{\bk}{\mathbf{K}}

\newcommand{\sT}{\mathsf{T}}

\newcommand{\alg}{\operatorname{alg}}

\newcommand{\CH}{\mathsf{CH}}

\newcommand{\Char}{\operatorname{char}}

\newcommand{\cl}{\operatorname{cl}}

\newcommand{\Corr}{\mathsf{Corr}}

\newcommand{\Eff}{\operatorname{Eff}}
\newcommand{\End}{\operatorname{End}}
\newcommand{\et}{{\textrm{\'et}}}

\newcommand{\Hom}{\operatorname{Hom}}
\newcommand{\id}{\operatorname{id}}

\newcommand{\isom}{\simeq}
\newcommand{\Ker}{\operatorname{Ker}}

\newcommand{\Nef}{\operatorname{Nef}}
\newcommand{\N}{\mathsf{N}}
\newcommand{\Nk}{\mathsf{N}^k(X)_\bR}

\newcommand{\num}{\equiv}

\newcommand{\PL}{\operatorname{PL}}
\newcommand{\PsEff}{\overline{\operatorname{Eff}}}

\newcommand{\pr}{\operatorname{pr}}

\newcommand{\rat}{\operatorname{rat}}

\newcommand{\Supp}{\operatorname{Supp}}

\newcommand{\tr}{\operatorname{tr}}
\newcommand{\Tr}{\operatorname{Tr}}

\newcommand{\Z}{\mathsf{Z}}

\begin{document}

%\title[Dynamical degrees of correspondences]{Dynamical degrees of certain correspondences of abelian varieties and of surfaces dominated by abelian surfaces}

%\title[Dynamical degrees of correspondences and semisimplicity]{Dynamical degree comparison conjecture and semisimplicity for abelian varieties and Kummer surfaces}

\title[A dynamical approach to generalized WRH and semisimplicity]{A dynamical approach to generalized Weil's Riemann hypothesis and semisimplicity}

\author{Fei Hu}
\address{Department of Mathematics, University of Oslo, P.O. Box 1053, Blindern, 0316 Oslo, Norway \endgraf
Department of Mathematics, Harvard University, Science Center 226, 1 Oxford Street, Cambridge, MA 02138, USA}

\email{\href{mailto:hf@u.nus.edu}{\tt hf@u.nus.edu}}
%\email{\href{mailto:fhu@math.uio.no}{\tt fhu@math.uio.no}}
%\urladdr{\url{https://sites.google.com/view/feihu90s/}}

\author{Tuyen Trung Truong}
\address{Department of Mathematics, University of Oslo, P.O. Box 1053, Blindern, 0316 Oslo, Norway}

\email{\href{mailto:tuyentt@math.uio.no}{\tt tuyentt@math.uio.no}}

\begin{abstract}
Let $X$ be a smooth projective variety over an algebraically closed field of arbitrary characteristic, and $f$ a dynamical correspondence of $X$.
In 2016, the second author conjectured that the dynamical degrees of $f$ defined by the pullback actions on $\ell$-adic \'etale cohomology groups and on numerical cycle class groups are equivalent, which we call the dynamical degree comparison (DDC) conjecture.
It contains the generalized Weil's Riemann hypothesis (for polarized endomorphisms) as a special case.

To proceed, we introduce the so-called \cref{conj:Gr}, which is a quantitative strengthening of the standard conjecture \hyperref[conj:C]{$C$} and holds on abelian varieties and Kummer surfaces.
We prove that for arbitrary varieties, \cref{conj:Gr} yields the generalized Weil's Riemann hypothesis.
Moreover, \cref{conj:Gr} plus the standard conjecture \hyperref[conj:D]{$D$} imply the so-called norm comparison (NC) conjecture, whose consequences include the DDC conjecture and the generalized semisimplicity conjecture (for polarized endomorphisms).
As an application, we obtain new results on the DDC conjecture for abelian varieties and Kummer surfaces, and the generalized semisimplicity conjecture for Kummer surfaces.
Finally, we also obtain a similar comparison result for effective finite correspondences of abelian varieties.
\end{abstract}

\subjclass[2020]{
14G17,	%Positive characteristic ground fields
37P25,	%Dynamical systems over finite ground fields
14K05,	%Abelian varieties and schemes, Algebraic theory
14J28,	%$K3$ surfaces and Enriques surfaces
14C25,	%Algebraic cycles
14F20. 	%Étale and other Grothendieck topologies and (co)homologies
}

%\date{\today}

\keywords{positive characteristic, dynamical degree, correspondence, algebraic cycle, \'etale cohomology, abelian variety, Kummer surface, Weil's Riemann hypothesis, standard conjectures, semisimplicity}

\thanks{The authors are supported by Young Research Talents grant \#300814 from the Research Council of Norway.}

\maketitle

%\tableofcontents

%%%%%%%%%%%%%%%%%%%%%%%%%%%%%%%%%%%%%%%%%%%%%%%%%%%%%%%%%%

\section{Introduction}
\label{section:intro}

%%%%%%%%%%%%%%%%%%%%%%%%%%%%%%%%%%%%%%%%%%%%%%%%%%%%%%%%%%

Unless otherwise stated, $\bk$ is an algebraically closed field of arbitrary characteristic and $X$ is a smooth projective variety of dimension $n$ over $\bk$.
A correspondence $f$ of $X$ is an algebraic cycle of codimension $n$ with rational coefficients, or its equivalence class, on $X\times X$.
There are natural pullback and pushforward actions of correspondences on Weil cohomology groups $H^{\bullet}(X)$ and cycle class groups $\Z^{\bullet}(X)_\bQ/\!\sim$ modulo certain equivalence relation $\sim$ (see \Cref{subsec:operation}).
Given two correspondences $f$, $g$ of $X$, their composition $g\circ f$ is a well-defined correspondence.
If moreover $f$ and $g$ are dynamical correspondences, which are natural generalizations of dominant rational maps (see \cref{def:dyn-corr}), we can define their dynamical composition $g\dyncomp f$, similar to the composition of dominant rational maps.
We denote by $f^{\circ m}$ and $f^{\Diamond m}$ the $m$-th (dynamical) iterates of $f$, respectively.
See \Cref{subsec:dyn-corr} for details.

\subsection{Conjectures in algebraic dynamics}

Throughout, we always choose $\ell$-adic \'etale cohomology as our Weil cohomology theory.
In other words, we let $H^{\bullet}(X) \coloneqq H^{\bullet}_{\et}(X, \bQ_\ell)$ with $\ell\neq \Char(\bk)$.
To a dynamical correspondence $f$ of $X$, we can assign two invariants as follows.

\begin{definition}[Cohomological \& numerical dynamical degrees]
\label{def:lambda-chi}
Fix a field isomorphism $\iota \colon \ol\bQ_\ell \isom \bC$ so that we may speak of the complex absolute value of an element of $\ol\bQ_\ell$.
Precisely, for any $\alpha \in \ol\bQ_\ell$, we set
\[
|\alpha|_{\iota} \coloneqq |\iota(\alpha)|.
\]
We endow the finite-dimensional $\bQ_\ell$-vector space $H^{\bullet}(X)$ with a norm $\norm{\cdot}_{\iota}$.
For any $0\le i\le 2n$, the {\it $i$-th cohomological dynamical degree} $\chi_i(f)_{\iota}$ of $f$ (with respect to $\iota$) is defined by
\begin{equation}
\label{eq:chi}
\chi_i(f)_{\iota} \coloneqq \limsup_{m\to \infty} \big\|{(f^{\Diamond m})^*|_{H^i(X)}}\big\|_{\iota}^{1/m}.
\end{equation}
On the other hand, for any $0\le k\le n$, we define the {\it $k$-{th} numerical dynamical degree} $\lambda_k(f)$ of $f$ as
\begin{equation}
\label{eq:lambda}
\lambda_k(f) \coloneqq \lim_{m\to \infty} \big\|{(f^{\Diamond m})^*|_{\Nk}}\big\|^{1/m},
\end{equation}
where we choose an arbitrary norm $\norm{\cdot}$ on the finite-dimensional real vector space
\[
\Nk \coloneqq \N^k(X) \otimes_\bZ \bR \coloneqq (\Z^k(X)/\!\num) \otimes_\bZ \bR.
\]
\end{definition}

\begin{remark}
\label{rmk:lambda-chi}
We note that in the above definition of $\lambda_k(f)$, the limit exists, independent of the choice of the norm on $\N^k(X)_\bR$, and is equal to the following limit
\[
\lim_{m\to \infty} ((f^{\Diamond m})^*H_X^k \cdot H_X^{n-k})^{1/m} = \lim_{m\to \infty} (f^{\Diamond m} \cdot \pr_1^*H_X^{n-k} \cdot \pr_2^*H_X^{k})^{1/m},
\]
where $H_X$ is an (arbitrary) ample divisor on $X$ and the $\pr_j$ are the projections from $X\times X$ to $X$; this is proved in \cite[Theorem~1.1(1)]{Truong20}. See also \cite[Theorems~1 and 2]{Dang20} for the case when $f$ is (the graph of) a dominant rational self-map of $X$.
For complex projective varieties (or more generally, compact K\"ahler manifolds), this was established earlier by Dinh and Sibony \cite{DS08}.
The difficulty in showing the existence of the limit lies in the fact that the dynamical composition does not necessarily commute with the pullback action.

We also note that the $\lambda_k(f)$ are defined intrinsically, while apparently, the $\chi_i(f)_{\iota}$ may depend on the field isomorphism $\iota$.
Another difference is worth mentioning: we now know that the $\lambda_k(f) < \infty$ are birational invariants (see \cite[Theorem~1.1(2)]{Truong20}), but it is still unknown if the $\chi_i(f)_{\iota}$ are always finite or birational invariants.
\end{remark}

Inspired by results in Complex Dynamics (see \cite{DS17} for a survey), in Esnault--Srinivas \cite{ES13}, and Weil's Riemann hypothesis (now Deligne's theorem \cite{Deligne74}), the second author proposed the following dynamical degree comparison conjecture (cf.~\cite[Question~2]{Truong}).

\begin{conjecture}[Dynamical Degree Comparison]
\label{conj:DDC}
Let $X$ be a smooth projective variety of dimension $n$ over $\bk$, and $f$ a dynamical correspondence of $X$.
Then $\chi_{2k}(f)_{\iota} = \lambda_k(f)$ for any $0\le k\le n$.
\end{conjecture}

Note that if we assume that the above \cref{conj:DDC} holds on $X$ and the product variety $X\times X$, then the celebrated Weil's Riemann hypothesis (now Deligne's theorem \cite{Deligne74}) follows.
Moreover, the following {\it generalized} Weil's Riemann hypothesis (which could be viewed as a positive-characteristic analog of a famous result on compact K\"ahler manifolds due to Serre \cite{Serre60}) also follows from \cref{conj:DDC} (for $X$ and $X\times X$).
See \cref{lemma:DDC-Serre}.

\begin{conjecture}[Generalized Weil's Riemann Hypothesis]
\label{conj:Serre}
Let $X$ be a smooth projective variety of dimension $n$ over $\bk$, and $f$ a self-morphism of $X$.
Suppose that $f$ is polarized, i.e., $f^*H_X \sim_{\rat} qH_X$ for an ample divisor $H_X$ on $X$ and a positive integer $q\in \bN_{\ge 2}$.
Then for any $0\le i \le 2n$, the eigenvalues of $f^*|_{H^i(X)}$ are $q$-Weil numbers of weight $i$, i.e., algebraic numbers $\alpha$ such that $|\sigma(\alpha)| = q^{i/2}$ for every embedding $\sigma\colon \ol\bQ\injmap \bC$.
\end{conjecture}

The above \cref{conj:Serre} follows from the standard conjectures (of Lefschetz type and Hodge type) mimicking Serre's argument \cite{Serre60}; see \cite[\S 4]{Kleiman68}.
Note, however, that in characteristic zero, not every variety admits a (non-isomorphic) polarized endomorphism (see \cite{NZ10,MZ18a}).
Another consequence of the standard conjectures is the following {\it generalized} semisimplicity for polarized endomorphisms acting on $\ell$-adic \'etale cohomology.

\begin{conjecture}[Generalized Semisimplicity]
\label{conj:SS}
Let $X$ and $f$ be as in \cref{conj:Serre}.
Then for any $0\le i \le 2n$, the pullback action $f^*|_{H^i(X)}$ of $f$ on $H^i(X)$ is diagonalizable over $\ol\bQ_\ell$.
\end{conjecture}

We conclude this subsection by introducing the so-called norm comparison conjecture, which is naturally motivated from Complex Dynamics (in fact, various norm estimates are crucial tools in proving certain fundamental results in dynamics; see, e.g., \cite{DS05a,DS08,Dang20,Truong20,DF21}).
We shall see that \cref{conj:NC} implies all above Conjectures~\ref{conj:DDC}, \ref{conj:Serre}, and \ref{conj:SS}.

\begin{conjecture}[Norm Comparison]
\label{conj:NC}
Let $X$ be a smooth projective variety of dimension $n$ over $\bk$.
Then for any dynamical correspondence $f$ of $X$, there is a constant $C>0$, independent of $f$, such that we have for any $0\le k\le n$,
\begin{equation}
\label{eq:Dinh-even}
\big\|f^*|_{H^{2k}(X)}\big\|_{\iota} \le C \, \big\|f^*|_{\Nk}\big\|,
\end{equation}
and for any $0\le k\le n-1$,
\begin{equation}
\label{eq:Dinh-odd}
\big\|f^*|_{H^{2k+1}(X)}\big\|_{\iota} \le C \sqrt{\big\|f^*|_{\Nk}\big\|\big\|f^*|_{\N^{k+1}(X)_\bR}\big\|}.
\end{equation}
\end{conjecture}

In characteristic zero, \cref{conj:NC} (and hence \cref{conj:DDC}) is a well-known fact in Complex Dynamics.
It actually follows from the Hodge decomposition, the fact that on complex projective manifolds the first Chern class of an ample line bundle is a K\"ahler class, and a result due to Dinh (see \cite[Proposition~5.8]{Dinh05}).
Also, \cref{conj:Serre,conj:SS}, for polarized endomorphisms of compact K\"ahler manifolds, are exactly Serre's famous result \cite{Serre60}.
In short, the positivity notion in cohomology plays a crucial role.

\begin{remark}
\label{rmk:DSS}
Below is the state of the art for \cref{conj:DDC,conj:Serre,conj:SS} in positive characteristic.
\begin{enumerate}[(1)]
\item \cref{conj:Serre} is known in the following two cases:
\begin{enumerate}[(i)]
\item $X$ is a curve or an abelian variety due to Weil's pioneering work \cite{Weil48};
\item $X$ is defined over a finite field and $f$ is the Frobenius endomorphism, i.e., the celebrated Weil's Riemann hypothesis, solved by Deligne \cite{Deligne74}.
\end{enumerate}

\item The generalized semisimplicity \cref{conj:SS} (for polarized endomorphisms) is only known for curves and abelian varieties again by Weil's work \cite{Weil48}.
Based on this, Deligne's result \cite{Deligne72} yields the semisimplicity of the Frobenius endomorphism of $K3$ surfaces.
However, \cref{conj:Serre,conj:SS} are still widely open for $K3$ surfaces, considering the standard conjecture of Hodge type remains open for squares of $K3$ surfaces.

\item In regards to \cref{conj:DDC}, there are some progresses for the case when $f$ is a self-morphism of $X$.
The case when $f$ is an automorphism of a smooth projective surface is essentially solved in \cite{ES13}.
In any dimension, for an arbitrary self-morphism $f$, using the spreading out and specialization argument, it has been established in \cite{Truong} that
\begin{equation}
\label{eq:max}
\max_{i=0, \ldots, 2n} \chi_i(f)_{\iota} = \max_{k=0, \ldots,n} \lambda_k(f).
\end{equation}
See also \cite{Shuddhodan19} for an alternative approach towards this equality using dynamical zeta functions.
As a consequence, if $f$ is a self-morphism of a surface with $\lambda_1(f)\ge \lambda_2(f)$ (this includes the case when $f$ is an automorphism of a surface), then \cref{conj:DDC} holds.
When $f$ is a self-morphism of an abelian variety, \cref{conj:DDC} is recently solved in \cite{Hu19,Hu-lc} using the structure theorem of endomorphism algebras of abelian varieties.

For the case of dominant rational self-maps, less is known about \cref{conj:DDC} in positive characteristic.
If we assume the standard conjecture \hyperref[conj:D]{$D$}, then \cref{eq:max} is obtained in \cite{Truong} for all dynamical correspondences $f$.
\end{enumerate}
\end{remark}

\subsection{Main results}
\label{subsec:main-results}

The main contribution of this article is to provide an alternative dynamical approach to \cref{conj:Serre,conj:SS}.
Our new dynamical input is the so-called \cref{conj:Gr}, which is a quantitative strengthening of the standard conjecture \hyperref[conj:C]{$C$}.
We shall see in \cref{lemma:Gr-holds-AV,lemma:Gr-Kummer} that \cref{conj:Gr} holds on abelian varieties and Kummer surfaces.

We first explain the notation $G_r$.
Let $r\in \bQ_{>0}$ be a positive rational number.
Then there always exists a unique homological correspondence $\gamma_{X,r}$ of $X$, i.e.,
\[
\gamma_{X,r} \in H^{2n}(X\times X) = \bigoplus_{i=0}^{2n} H^i(X)\otimes H^{2n-i}(X) \isom \bigoplus_{i=0}^{2n} \End_{\bQ_\ell}(H^i(X)),
\]
so that the corresponding pullback action $\gamma_{X,r}^*$ on $H^i(X)$ is the multiplication-by-$r^i$ map for each $i$.
Note that $\gamma_{X,r}$ commutes with all homological correspondences of $X$.
We also denote it by $\gamma_r$ if there is no confusion.

\begin{conjGr}
\label{conj:Gr}
Let $X$ be a smooth projective variety of dimension $n$ over $\bk$.
Then for any $r\in \bQ_{>0}$, the above homological correspondence $\gamma_{r}$ of $X$ is algebraic and represented by a rational algebraic $n$-cycle $G_r$ on $X\times X$, i.e., $\gamma_r = \cl_{X\times X}(G_r)$.
Moreover, for any dynamical correspondence $f$ of $X$, there is a constant $C>0$, independent of $r$ and $f$, such that
\begin{equation}
\label{eq:Gr}
\|G_r\circ f\| \le C \deg(G_r\circ f),
\end{equation}
where $\norm{\cdot}$ denotes a norm on $\N^n(X\times X)_\bR$ and $\deg(\cdot)$ is the degree function on $X\times X$ with respect to a fixed ample divisor $H_{X\times X}$.
\end{conjGr}

\begin{remark}
\label{rmk:Gr}
\begin{enumerate}[(1)]
\item Note that the above constant $C$ does depend on the choices of the norm and the degree function, but is independent of $r$ and $f$.
Since all norms on $\N^n(X\times X)_\bR$ are equivalent, in practice, it suffices to adopt one specific norm in verifying \cref{conj:Gr}.

\item If the standard conjecture \hyperref[conj:C]{$C$} holds, then $\gamma_{r}$ can be represented by $G_r \coloneqq \sum_{i=0}^{2n} r^i\Delta_i$, where $\Delta_i \in \Z^n(X\times X)_\bQ$ corresponds the $i$-th K\"unneth component $\pi_i$ of the diagonal class $\cl_{X\times X}(\Delta_X)$.
Moreover, by the proof of \cite[Theorem~2A11]{Kleiman68} due to Lieberman, the standard conjecture \hyperref[conj:C]{$C$} is equivalent to the algebraicity of $\gamma_r$ for every (or rather, for $2n+1$ different) $r\in \bQ_{>0}$.
Over finite fields, the standard conjecture \hyperref[conj:C]{$C$} holds; further, the algebraic cycles $\Delta_i$ (representing the K\"unneth components $\pi_i$) and hence $G_r$ are independent of Weil cohomology theory (see \cite[Theorem~2]{KM74}).
In this case, \cref{conj:Gr} simply refers to the inequality \eqref{eq:Gr}, which is also independent of Weil cohomology theory.
\end{enumerate}
\end{remark}

We always choose $H_{X\times X} \coloneqq \pr_1^*H_X + \pr_2^*H_X$.
Then the degree $\deg(f)$ of a correspondence $f$ of $X$ is equal to
\[
f\cdot H_{X\times X}^n = \sum_{i=0}^n \binom{n}{i} f\cdot \pr_1^*H_X^{n-i} \cdot \pr_2^*H_X^{i} = \sum_{i=0}^n \binom{n}{i} f^*H_X^{i} \cdot H_X^{n-i}.
\]
In the summation, each $f^*H_X^{i} \cdot H_X^{n-i}$ is called the {\it $i$-th degree of $f$}, denoted by $\deg_i(f)$.
It is worth noting that for an effective correspondence $f$ of $X$, albeit the composite correspondence $G_r\circ f$ in \cref{conj:Gr} may not be effective anymore, we always have
\[
\deg(G_r \circ f) = \sum_{i=0}^n \binom{n}{i} (G_r\circ f)^*H_X^{i} \cdot H_X^{n-i} = \sum_{i=0}^{n} \binom{n}{i} r^{2i} \deg_i(f) \ge 0.
\]

Our first main result shows that \cref{conj:Gr} implies the generalized Weil’s Riemann hypothesis (i.e., \cref{conj:Serre}).

\begin{theorem}
\label{thm:B}
Suppose that Conjecture~\hyperref[conj:Gr]{$G_r(X)$} holds.
Then the following statements hold.
\begin{enumerate}[{\em (1)}]
\item \label{Assertion:B-1}
There is a constant $C>0$ such that for any dynamical correspondence $f$ of $X$, we have for any $0\le k\le n$,
\begin{equation}
\label{eq:tr-deg-even}
|\Tr(f^*|_{H^{2k}(X)})| \le C \deg_k(f),
\end{equation}
and for any $0\le k\le n-1$,
\begin{equation}
\label{eq:tr-deg-odd}
|\Tr(f^*|_{H^{2k+1}(X)})| \le C \sqrt{\deg_k(f)\deg_{k+1}(f)}.
\end{equation}

\item \label{Assertion:B-2}
Let $f$ be a dynamical correspondence of $X$ such that the $m$-th iterate $f^{\circ m}$ is still a dynamical correspondence for any $m\in \bN_{>0}$.
Then we have for any $0\le k\le n$,
\begin{equation}
\label{eq:rho-even}
\rho(f^*|_{H^{2k}(X)}) = \rho(f^*|_{\N^k(X)_\bR}),
\end{equation}
and for any $0\le k\le n-1$,
\begin{equation}
\label{eq:rho-odd}
\rho(f^*|_{H^{2k+1}(X)}) \le \sqrt{\rho(f^*|_{\N^k(X)_\bR}) \, \rho(f^*|_{\N^{k+1}(X)_\bR})},
\end{equation}
where $\rho(\varphi)$ denotes the spectral radius of a linear transformation $\varphi$.
In particular, \cref{conj:Serre} holds on $X$.
\end{enumerate}
\end{theorem}

\begin{remark}
\label{rmk:B}
\begin{enumerate}[(1)]
\item It is worth mentioning that to solve \cref{conj:Serre} we only need \cref{conj:Gr} for (graphs of) polarized endomorphisms in \cref{thm:B}\eqref{Assertion:B-2}.
The program by Bombieri and Grothendieck is to use the standard conjectures of Lefschetz type and Hodge type to solve \cref{conj:Serre} (see \cite[\S 4]{Kleiman68}).
However, even for abelian varieties, the standard conjecture of Hodge type is currently known only up to dimension four \cite{Ancona21}, while our \cref{conj:Gr} holds in all dimensions (see \cref{lemma:Gr-holds-AV}).

\item Note that under the assumption of \cref{thm:B}, the standard conjecture \hyperref[conj:C]{$C$} holds by \cite[Proof of Theorem~2A11]{Kleiman68}.
Hence both $\Tr(f^*|_{H^{i}(X)})$ and $\rho(f^*|_{H^{i}(X)})$ are independent of the choice of $\iota \colon \ol\bQ_\ell \isom \bC$ by the Lefschetz trace formula (see \cref{prop:Lefschetz-trace-formula}), this is why we suppress the subscript $\iota$.
\end{enumerate}
\end{remark}

To deal with arbitrary dynamical correspondences, we prove \cref{conj:NC} assuming the above \cref{conj:Gr} and the standard conjecture \hyperref[conj:Dk]{$D$}.
The theorem below gives various logical relationships between Conjectures~\ref{conj:Gr} and \hyperref[conj:D]{$D$}, Conjectures~\ref{conj:DDC}, \ref{conj:Serre}, \ref{conj:SS}, and \ref{conj:NC}.

\begin{theorem}
\label{thm:C}
The following statements hold.
\begin{enumerate}[{\em (1)}]
\item \label{Assertion:C-1}
Conjectures~\hyperref[conj:Gr]{$G_r(X)$} and \hyperref[conj:Dk]{$D^n(X\times X)$} imply \cref{conj:NC}.

\item \label{Assertion:C-2}
\cref{conj:NC} implies \cref{conj:DDC}.

\item \label{Assertion:C-3}
\cref{conj:DDC} holds on $X$ and $X\times X$ implies that \cref{conj:Serre} holds on $X$.

\item \label{Assertion:C-4}
\cref{conj:NC} implies \cref{conj:SS}.
\end{enumerate}
\end{theorem}

\begin{remark}
\label{rmk:C}
\begin{enumerate}[(1)]
\item As mentioned earlier, in characteristic zero, \cref{conj:NC} is a well-known fact in Complex Dynamics.
Thus the combination of \cref{conj:Gr} with the standard conjecture \hyperref[conj:D]{$D$} may replace the positivity notions in K\"ahler manifolds (such as positive closed smooth forms and currents) for our purposes, in particular, provide affirmative answers to \cref{conj:DDC,conj:Serre,conj:SS}.
We expect that \cref{conj:Gr} could later find more applications for analogs in positive characteristics of results currently only known in characteristic zero.

\item In summary, we have the following logical relationships:
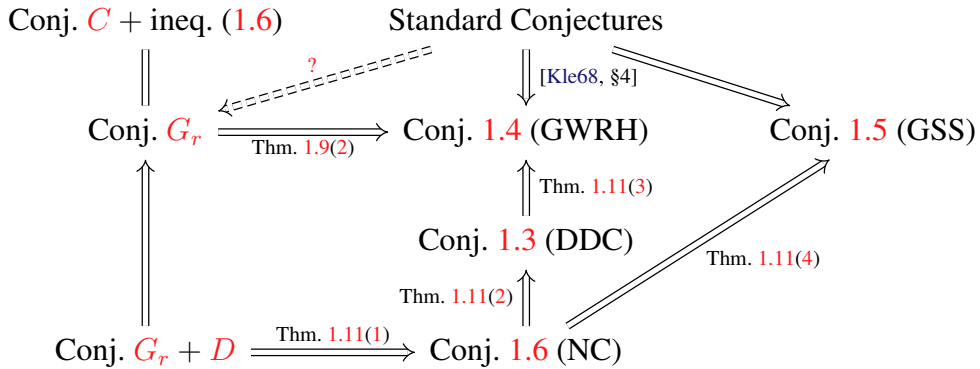
\[
\begin{tikzcd}
\textrm{Conj.~\hyperref[conj:C]{$C$}} + \textrm{ineq.~\eqref{eq:Gr}} \arrow[d, equal] & \textrm{Standard Conjectures} \arrow[d, Rightarrow, "\textrm{\cite[\S 4]{Kleiman68}}"] \arrow[dr, Rightarrow] \arrow[dl, Rightarrow, dashed, "\textrm{\color{red} ?}"'] & \\
\textrm{Conj.~\ref{conj:Gr}} \arrow[r, Rightarrow, "\textrm{Thm.~\ref{thm:B}\eqref{Assertion:B-2}}"'] & \textrm{Conj.~\ref{conj:Serre} (GWRH)} & \textrm{Conj.~\ref{conj:SS} (GSS)} \\
 & \textrm{Conj.~\ref{conj:DDC} (DDC)} \arrow[u, Rightarrow, "\textrm{~Thm.~\ref{thm:C}\eqref{Assertion:C-3}}"'] & \\
\textrm{Conj.~\ref{conj:Gr}}+\textrm{\hyperref[conj:D]{$D$}} \arrow[uu, Rightarrow] \arrow[r, Rightarrow, "\textrm{Thm.~\ref{thm:C}\eqref{Assertion:C-1}}"] & \textrm{Conj.~\ref{conj:NC} (NC)} \arrow[u, Rightarrow, "\textrm{Thm.~\ref{thm:C}\eqref{Assertion:C-2}~}"] \arrow[uur, Rightarrow, "\textrm{Thm.~\ref{thm:C}\eqref{Assertion:C-4}}"'] & 
\end{tikzcd}
\]
(Memory aids: GWRH is for generalized Weil's Riemann hypothesis, GSS for generalized semisimplicity, DDC for dynamical degree comparison, and NC for norm comparison.)
One may wonder the relationship between the standard conjectures and our \cref{conj:Gr}.
We recently show that indeed for polarized endomorphisms, \cref{conj:Gr} is a consequence of the standard conjectures (see \cite{HT-pol}).
We suspect that this is true for all effective correspondences (namely, the standard conjectures imply \cref{conj:Gr}).
\end{enumerate}
\end{remark}

As an application, we deduce several consequences from the above \cref{thm:B,thm:C}.
Among other things, we show that DDC \cref{conj:DDC} holds for algebraically stable dynamical correspondences (see \cref{def:alg-stable}) of abelian varieties or Kummer surfaces\footnote{By definition, a Kummer surface is the minimal resolution of the quotient $A/\langle -1\rangle$ of an abelian surface $A$ over $\bk$ with $\Char(\bk)\neq 2$ by the sign involution.}, extending the main results in \cite{Hu19,Hu-lc}.
We also obtain the semisimplicity of polarized endomorphisms of Kummer surfaces over finite fields.
Note that \cref{conj:Serre} for abelian varieties is well known by Weil \cite{Weil48}; see \cref{rmk:DSS} for more details.
We thus recover this classical result using our new method.
%Moreover, with the aid of \cite{Clozel99}, we prove that DDC \cref{conj:DDC} holds on abelian varieties over finite fields for infinitely many primes $\ell\neq p$.

\begin{theorem}
\label{thm:A}
The following statements hold.
\begin{enumerate}[{\em (1)}]
\item \label{Assertion:A-1}
Let $f$ be an algebraically stable dynamical correspondence of an abelian variety or a Kummer surface $X$ over $\bk$.
Then we have for any $0\le k\le \dim X$,
\[
\chi_{2k}(f)_{\iota} = \lambda_k(f),
\]
and for any $0\le k\le \dim X - 1$,
\[
\chi_{2k+1}(f)_{\iota} \le \sqrt{\lambda_k(f)\lambda_{k+1}(f)}.
\]
In particular, \cref{conj:Serre} holds on abelian varieties and Kummer surfaces.

\item \label{Assertion:A-2}
\cref{conj:NC,conj:DDC,conj:SS} hold on Kummer surfaces over $\ol\bF_p$.

\item \label{Assertion:A-3} \cref{conj:NC,conj:DDC} hold on abelian varieties over $\ol\bF_p$ for infinitely many primes $\ell\neq p$.
\end{enumerate}
\end{theorem}

Our last result concerns a dynamical perspective of finite correspondences, which are the building block in constructions of various triangulated categories of (effective) motives over $\bk$ by Suslin and Voevodsky (see \cite{MVW06}).
Typical examples of them include graphs of (not necessarily surjective) morphisms, transposes of graphs of finite surjective morphisms, and positive combinations of such; see \Cref{subsec:finite-corr} for details.
In particular, it also extends the main results of \cite{Hu19,Hu-lc}.

\begin{theorem}
\label{thm:D}
Let $X$ be an abelian variety of dimension $n$ over $\bk$, and $f$ an effective finite correspondence of $X$ (see \cref{def:finite-corr}).
Then we have for any $0\le k\le n$,
\begin{equation*}
%\label{eq:A-even}
\rho(f^*|_{H^{2k}(X)}) = \rho(f^*|_{\N^k(X)_\bR}),
\end{equation*}
and for any $0\le k\le n - 1$,
\begin{equation*}
%\label{eq:A-odd}
\rho(f^*|_{H^{2k+1}(X)}) \le \sqrt{\rho(f^*|_{\N^k(X)_\bR}) \, \rho(f^*|_{\N^{k+1}(X)_\bR})}.
\end{equation*}
\end{theorem}

\subsection{Outline}

In \Cref{section:prelim}, we will provide preliminaries on correspondences, including the classical way to compose them, how they act on cycle groups and cohomology, and various numerical invariants associated to them.
Besides, we recall the standard conjectures \hyperref[conj:C]{$C$} and \hyperref[conj:D]{$D$} in \Cref{subsec:SCCD} and Lieberman's \cref{lemma:Lieberman} in \Cref{subsec:operation}.
We also present a useful \cref{lemma:boundedness} on the boundedness of self-intersections of correspondences in \Cref{subsec:boundedness}, which might be of its independent interest.
In \Cref{subsec:dyn-corr}, we discuss dynamical correspondences, their compositions, and dynamical pullback, which may be less familiar than their counterparts in intersection theory.
In \Cref{subsec:finite-corr}, we will explain the basic properties of effective finite correspondences.
The section thereafter is devoted to proofs of main \cref{thm:A,thm:B,thm:C,thm:D}.

\subsection{Acknowledgments}

We would like to heartily thank Bruno Kahn and Paul Arne \O stv\ae r for referring us to the notion of finite correspondence, which helps to strengthen some results.
We would also like to thank Giuseppe Ancona, Nguyen-Bac Dang, Charles Favre, Claire Voisin, Duc-Viet Vu, Yuri Zarhin, and De-Qi Zhang for inspiring comments and discussions.

\section{Preliminaries on correspondences}
\label{section:prelim}

Let $\bk$ be an algebraically closed field of arbitrary characteristic.
Let $X$ be a smooth projective variety of dimension $n$ over $\bk$.
Choose $\ell$-adic \'etale cohomology as our Weil cohomology theory, i.e., $H^{\bullet}(X) \coloneqq H_{\et}^{\bullet}(X,\bQ_\ell)$ with $\ell\neq \Char(\bk)$.
In particular, we have a cup product $\cup$, Poincar\'e duality, the K\"unneth formula, the cycle class map $\cl_X$, the weak Lefschetz theorem, and the hard Lefschetz theorem.
We refer to Fulton's book \cite{Fulton98} for intersection theory.

\subsection{Algebraic cycles}
\label{subsec:cycle}

Let $\Z^k(X)$ be the abelian group generated by integral algebraic cycles of codimension $k$ on $X$.
Let $\N^k(X)$ denote the quotient group of $\Z^k(X)$ by numerical equivalence $\num$, which turns out to be a finitely generated abelian group (see \cite[Theorem~3.5]{Kleiman68}).
For a field $\bF$ of characteristic zero (e.g., $\bQ$, $\bQ_\ell$, $\bR$, or $\bC$), set $\N^k(X)_\bF \coloneqq \N^k(X) \otimes_\bZ \bF$.

Fix an ample divisor $H_X$ on $X$.
For an algebraic cycle $Z$ of codimension $k$ on $X$, we call $Z\cdot H_X^{n-k}$ the {\it degree} of $Z$, denoted by $\deg_{H_X}(Z)$ (or simply, $\deg(Z)$ as the ample divisor is clear from the context).
By definition, the degree $\deg(Z)$ of $Z$ is a numerical invariant of $Z$, i.e., depending only the numerical class $[Z]\in \Nk$ of $Z$.

A numerical class $\alpha \in \Nk$ is {\it effective}, it $\alpha = [Z]$ for some effective cycle $Z\in \Z^k(X)_\bR$.
Note that by the Nakai--Moishezon criterion, the degree of an effective cycle (class) is always nonnegative.
The {\it effective cone} $\Eff^k(X)$ of $X$ denotes the convex cone in $\Nk$ generated by all numerical classes of effective cycles.
The closure of $\Eff^k(X)$ in $\Nk$ is called the {\it pseudoeffective cone} $\PsEff^k(X)$ of $X$.

Note that $\Nk$ is a finite-dimensional real vector space.
We have several ways to endow it with norms, equivalent though.
For instance, given any $\alpha \in \Nk$, we define
\begin{equation}
\label{eq:cycle-norm}
\|\alpha\| \coloneqq \inf \big\{\deg(\alpha^+ + \alpha^-) : \alpha \equiv \alpha^+ - \alpha^-, \alpha^\pm \in \Eff^k(X) \big\}.
\end{equation}
This is indeed a norm on $\Nk$ (see, e.g., \cite[\S2.4]{Truong20}).

\begin{remark}
\label{rmk:norm-effective}
We note that if $\alpha \in \Eff^k(X)$, then $\|\alpha\| = \deg(\alpha)$.
First, $\|\alpha\| \le \deg(\alpha)$ is clear. If we write $\alpha \equiv \alpha^+ - \alpha^-$ for some $\alpha^\pm \in \Eff^k(X)$, then $\deg(\alpha^+ + \alpha^-) \ge \deg(\alpha)$.
Taking infimum, we obtain the converse.
\end{remark}

On the other hand, because of the duality between $\Nk$ and $\N^{n-k}(X)_\bR$, we can speak of the operator norm on $\Nk$ as linear operators on $\N^{n-k}(X)_\bR$: for any $\alpha \in \Nk$,
\begin{equation}
\label{eq:cycle-dual-norm}
\|\alpha\|' \coloneqq \sup \big\{|\alpha \cdot \beta| : \beta \in \N^{n-k}(X)_\bR, \|\beta\| = 1 \big\}.
\end{equation}

\subsection{Standard Conjectures \texorpdfstring{$C$}{C} and \texorpdfstring{$D$}{D}}
\label{subsec:SCCD}

For the readers' convenience, we recall the standard conjectures \hyperref[conj:C]{$C$} and \hyperref[conj:D]{$D$}.
We refer the reader to Kleiman's survey \cite{Kleiman68} for a comprehensive exposition on all standard conjectures and to the book \cite{MNP13} for a modern account.

Let $\Delta_X \subset X \times X$ be the diagonal.
Then by the K\"unneth formula
\[
H^{2n}(X\times X) = \bigoplus_{i=0}^{2n} H^{i}(X) \otimes H^{2n-i}(X),
\]
we have the following decomposition of the diagonal class:
\[
\cl_{X\times X}(\Delta_X) = \sum_{i=0}^{2n} \pi_{i},
\]
where $\pi_{i} \in H^{i}(X) \otimes H^{2n-i}(X)$ for each $i$.
Note that the pullback $\pi_i^*$ of the K\"unneth component $\pi_i$ is exactly the $i$-th projection operator $\pi_i \colon H^{\bullet}(X) \to H^i(X)$, hence the name.\footnote{However, in \cite{Kleiman68,MNP13}, they are all using the pushforward of the K\"unneth components to give the projections, which is merely a conventional discrepancy.}
Below is the standard conjecture of K\"unneth type.

\begin{conjC}%[Conjecture of K\"unneth type]
\label{conj:C}
The K\"unneth components $\pi_{i}$ of the diagonal class $\cl_{X\times X}(\Delta_X)$ are algebraic. Namely, there are rational algebraic cycles $\Delta_{i} \in \Z^n(X\times X)_\bQ$ such that $\pi_{i} = \cl_{X\times X}(\Delta_{i})$.
\end{conjC}

\begin{remark}
\label{rmk:ABC}
\begin{enumerate}[(1)]
\item \cref{conj:C} holds if $X$ is a curve, a surface (see \cite[Corollary~2A10]{Kleiman68}), or an abelian variety by Lieberman (see, e.g., \cite[Theorem~2A11]{Kleiman68}). Over $\bC$, in general, \cref{conj:C} follows from the famous Hodge conjecture.

\item If $X$ is defined over a finite field $\bF_q$, then \cref{conj:C} holds as a consequence of Weil's Riemann hypothesis (see \cite[Theorem~2]{KM74}).
\end{enumerate}
\end{remark}

\begin{conjDk}
\label{conj:Dk}
Homological equivalence $\sim_{\hom}$ of algebraic cycles of codimension $k$ on $X$ coincides with numerical equivalence $\num$.
Precisely, for any rational algebraic cycle $Z\in \Z^k(X)_\bQ$, if $Z \num 0$ in $\N^k(X)_\bQ$, then $\cl_X(Z) = 0$ in $H^{2k}(X)$.
\end{conjDk}

\begin{conjD}
\label{conj:D}
\cref{conj:Dk} hold for all $0\le k\le n$.
\end{conjD}

\begin{remark}
\label{rmk:D}
\begin{enumerate}[(1)]
\item Conjecture~\hyperref[conj:Dk]{$D^1(X)$} holds (see Matsusaka \cite{Matsusaka57}).

\item In characteristic zero, \cref{conj:D} holds for abelian varieties by Lieberman.

\item For abelian varieties over finite fields of prime characteristic $p$, Clozel \cite{Clozel99} proved that there are infinitely many primes $\ell\neq p$ (in fact, a set of primes of positive density), such that \cref{conj:D} holds for $\ell$-adic \'etale cohomology.
Based on this, recently, Ancona \cite{Ancona21} extends Clozel's result above to arbitrary fields in positive characteristic.
\end{enumerate}
\end{remark}

Denote by $H_{\alg}^{2k}(X)$ the $\bQ_\ell$-vector subspace of $H^{2k}(X)$ generated by all algebraic cohomology classes.
By Poincar\'e duality, one can define the "transcendental" part of $H^{2k}(X)$ by
\begin{equation*}
%\label{eq:transcendental-part}
H^{2k}_{\tr}(X) \coloneqq H^{2n-2k}_{\alg}(X)^{\perp} \coloneqq \{ \alpha \in H^{2k}(X) : \alpha \cup \beta = 0, \textrm{ for all } \beta\in H^{2n-2k}_{\alg}(X) \}.
\end{equation*}
In other words, $H^{2k}_{\tr}(X)$ is the kernel of the following natural linear map:
\begin{equation}
\label{eq:H2k-tr}
H^{2k}(X) \lra \Hom_{\bQ_\ell}(H^{2n-2k}_{\alg}(X), \bQ_\ell), \quad \alpha \mapsto (\beta \mapsto \alpha \cup \beta).
\end{equation}

\begin{lemma}
\label{lemma:decomposition}
\cref{conj:Dk} holds if and only if there is a direct sum decomposition of $H^{2k}(X)$ as follows:
\begin{equation}
\label{eq:direct-sum-decomposition}
H^{2k}(X) = H^{2k}_{\alg}(X) \oplus H^{2k}_{\tr}(X).
\end{equation}
\end{lemma}
\begin{proof}
We first show the "if" part.
Suppose that we have the above direct sum decomposition \eqref{eq:direct-sum-decomposition}; in particular, $H^{2k}_{\alg}(X) \cap H^{2k}_{\tr}(X) = 0$.
Let $Z \in \Z^k(X)_\bQ$ such that $Z\num 0$.
Clearly, $\cl_X(Z) \in H^{2k}_{\alg}(X)$.
On the other hand, choose $Z'_1,\ldots,Z'_m \in \Z^{n-k}(X)_\bQ$ such that their cohomology classes form a $\bQ_\ell$-basis of $H^{2n-2k}_{\alg}(X)$.
By assumption, $\cl_X(Z) \cup \cl_X(Z'_i) = Z\cdot Z'_i = 0$ for all $i$.
It thus follows that $\cl_X(Z) \in H^{2k}_{\tr}(X)$ and hence $\cl_X(Z) = 0$, which shows that \cref{conj:Dk} holds.

Conversely, suppose that \cref{conj:Dk} (equivalently, \hyperref[conj:Dk]{$D^{n-k}(X)$}) holds.
Then by \cite[Proposition~3.6]{Kleiman68}, one has $H^{2k}_{\alg}(X) = \N^k(X) \otimes_\bZ \bQ_\ell$ and $H^{2n-2k}_{\alg}(X) = \N^{n-k}(X) \otimes_\bZ \bQ_\ell$.
Let $Z_1,\ldots,Z_m \in \Z^{k}(X)_\bQ$ and $Z'_1,\ldots,Z'_m \in \Z^{n-k}(X)_\bQ$ be such that their cohomology classes form $\bQ_\ell$-bases of $H^{2k}_{\alg}(X)$ and $H^{2n-2k}_{\alg}(X)$, respectively.
We may assume further that they are dual to each other, i.e., $Z_i\cdot Z'_j = \delta_{ij}$.
Now, we let $\alpha = \sum_i a_i \cl_X(Z_i) \in H^{2k}_{\alg}(X)$ with $a_i\in \bQ_\ell$.
If $\alpha \in H^{2k}_{\tr}(X)$ also, then by definition, one has $0 = \alpha \cup \cl_X(Z'_j) = \sum_i a_i \cl_X(Z_i)\cup \cl_X(Z'_j) = \sum_i a_i Z_i\cdot Z'_j = a_j$ for all $j$, which implies that $\alpha = 0$.
So we prove that the right-hand side of \cref{eq:direct-sum-decomposition} is a direct sum.
Since $H^{2k}_{\tr}(X)$ is the kernel of the homomorphism \eqref{eq:H2k-tr}, it follows that
\begin{align*}
\dim_{\bQ_\ell} H^{2k}(X) &\le \dim_{\bQ_\ell} H^{2k}_{\tr}(X) + \dim_{\bQ_\ell} \Hom_{\bQ_\ell}(H^{2n-2k}_{\alg}(X), \bQ_\ell) \\
&= \dim_{\bQ_\ell} H^{2k}_{\tr}(X) + \dim_{\bQ_\ell} H^{2k}_{\alg}(X) \\
&= \dim_{\bQ_\ell}(H^{2k}_{\tr}(X) \oplus H^{2k}_{\alg}(X)).
\end{align*}
We thus obtain the direct sum decomposition \eqref{eq:direct-sum-decomposition} and conclude the proof of \cref{lemma:decomposition}.
\end{proof}

\subsection{Correspondences}
\label{subsec:corr}

Correspondences are natural generalizations of graphs of morphisms or rational maps.
They were used essentially by Weil in his proof of Riemann hypothesis for curves over finite fields.
They also play an important role in the construction of motives (see \cite{MNP13}).
We begin with the definition of correspondences and refer to \cite[Chapter~16]{Fulton98} for details.

\begin{definition}
\label{def:corr}
Let $X$ and $Y$ be smooth projective varieties of dimension $n$ over $\bk$.
Throughout, unless otherwise stated, a {\it correspondence} $f$ from $X$ to $Y$ is a rational algebraic cycle of codimension $n$, or its equivalence class, on $X \times Y$, i.e., $f\in \Z^n(X\times Y)_\bQ$ or $\Z^n(X\times Y)_\bQ/\!\sim$, where $\sim$ is an adequate equivalence relation (e.g., rational equivalence $\sim_{\rat}$, homological equivalence $\sim_{\hom}$, or numerical equivalence $\num$; see \cite[\S1.2]{MNP13} for a discussion on various adequate equivalence relations).
We shall write $f \colon X \vdash Y$ to denote a correspondence $f$ from $X$ to $Y$.
We also use the following shorthand notation:
\begin{align*}
\Corr^0(X,Y) &\coloneqq \Z^n(X\times Y)_\bQ, \\
\Corr^0_{\sim}(X,Y) &\coloneqq \Z^n(X\times Y)_\bQ/\!\sim.
\end{align*}
\end{definition}

The canonical {\it transpose} $f^{\sT} \colon Y\vdash X$ of a correspondence $f \colon X\vdash Y$ is defined as the pushforward of $f$ under the involution map $\tau \colon X\times Y \to Y\times X$ which interchanges the coordinates.

A correspondence $f \colon X\vdash Y$ is {\it effective}, if it can be represented by an effective rational $n$-cycle on $X\times Y$.
It is further called {\it irreducible}, if $f$ is represented by an (irreducible) algebraic subvariety of $X\times Y$ of codimension $n$.

\begin{example}
Let $\pi \colon X \to Y$ be a morphism of smooth projective varieties of dimension $n$ over $\bk$.
Then its graph $\Gamma_\pi$, defined by the image of the diagonal $\Delta_X$ under the product morphism $\id_X \times \pi \colon X \times X \to X \times Y$, is an irreducible algebraic cycle of codimension $n$ on $X\times Y$, hence an irreducible correspondence from $X$ to $Y$.
Its transpose is thus denoted by $\Gamma_\pi^\sT$ (or $\pi^\sT$).
Similarly, one can consider the graph $\Gamma_\pi$ of a rational map $\pi \colon X \ratmap Y$, which is also an irreducible correspondence from $X$ to $Y$.
\end{example}

\subsection{Composition of correspondences}
\label{subsec:comp}

Correspondences can be naturally composed in intersection theory (see, e.g., \cite[\S16.1]{Fulton98}).
More precisely, given two arbitrary correspondences $f \colon X\vdash Y$ and $g \colon Y\vdash Z$, the composite correspondence, denoted by $g\circ f$, is defined by
\begin{equation*}
%\label{eq:comp}
g\circ f \coloneqq \pr_{13,*}(\pr_{12}^*f \cdot \pr_{23}^*g),
\end{equation*}
where the $\pr_{ij}$ denote the natural projections from $X\times Y\times Z$ to the appropriate factors, respectively.
Note that the intersection is in $\Z^{2n}(X\times Y\times Z)_\bQ/\!\sim$, and hence the composite correspondence $g\circ f \in \Corr^0_{\sim}(X,Z)$.
In particular, this makes $\Corr^0_{\sim}(X,X)$ a (not necessarily commutative) ring. Clearly, the identity is given by the class of the diagonal $\Delta_X$.
To distinguish the so-called dynamical composition of dynamical correspondences that will be introduced in \Cref{subsec:dyn-corr}, we simply call $g\circ f$ the {\it composition} of $f$ followed by $g$.
For a correspondence $f\colon X\vdash X$ and any $m\in \bN_{>0}$, we denote by $f^{\circ m}$ the {\it $m$-th iterate} of $f$.

\begin{remark}
\label{rmk:composition-morphism}
It is worth mentioning that given two morphisms $\phi \colon X\to Y$ and $\psi \colon Y\to Z$ of smooth projective varieties of dimension $n$ over $\bk$, we always have that $\Gamma_\psi \circ \Gamma_\phi = \Gamma_{\psi \circ \phi}$ as correspondences from $X$ to $Z$, where $\psi \circ \phi$ denotes the standard composition of morphisms.
However, if both $\phi$ and $\psi$ are endomorphisms of an abelian variety, then in general, we do {\it not} have that $\Gamma_\phi + \Gamma_\psi = \Gamma_{\phi + \psi}$.
\end{remark}

\subsection{Operations of correspondences on cycle groups and cohomology}
\label{subsec:operation}

We first recall the definition of the natural pullback and pushforward actions of correspondences on algebraic cycle class groups and cohomology groups.

\begin{definition}
\label{def:pull-push-corr}
Let $f \colon X\vdash Y$ be a correspondence from $X$ to $Y$.
We define the pullback $f^*$ and pushforward $f_*$ on the algebraic cycle class groups as follows:
\begin{align*}
f_* &\colon \Z^k(X)_\bQ/\!\sim \, \lra \Z^k(Y)_\bQ/\!\sim, \quad \alpha \mapsto \pr_{2,*}(f \cdot \pr_1^*{\alpha}), \\
f^* &\colon \Z^k(Y)_\bQ/\!\sim \, \lra \Z^k(X)_\bQ/\!\sim, \quad \beta \mapsto \pr_{1,*}(f \cdot \pr_2^*{\beta}),
\end{align*}
where the $\pr_i$ denote the natural projections from $X\times Y$ to $X$ and $Y$, respectively.

Similarly, if $f \in\Corr^0_{\sim}(X,Y)$ for an equivalence relation $\sim$ finer than, or equal to, homological equivalence relation $\sim_{\hom}$, we can define natural pullback $f^*$ and pushforward $f_*$ on the cohomology groups $H^i(X)$ as follows:
\begin{align*}
f_* &\colon H^i(X) \lra H^i(Y), \quad \alpha \mapsto \pr_{2,*}(\cl_{X\times Y}(f) \cup \pr_1^*{\alpha}), \\
f^* &\colon H^i(Y) \lra H^i(X), \quad \beta \mapsto \pr_{1,*}(\cl_{X\times Y}(f) \cup \pr_2^*{\beta}),
\end{align*}
where $\cl_{X\times Y} \colon \Z^k(X\times Y)_\bQ \lra H^{2k}(X\times Y)$ is the cycle class map, which factors through $\Z^k(X\times Y)_\bQ/\!\sim$ by assumption.
\end{definition}

\begin{remark}
The correspondences with respect to rational equivalence operate on the rational Chow groups $\CH^{\bullet}(-)_\bQ$, on cohomology groups $H^{\bullet}(-)$, and on the numerical cycle class groups $\N^{\bullet}(-)_\bQ$ or $\N^{\bullet}(-)_\bR$.
The correspondences with respect to homological equivalence operate on cohomology groups and on the numerical cycle class groups.
Note, however, that the correspondences with respect to numerical equivalence operate only on cohomology groups provided the standard conjecture \hyperref[conj:D]{$D$} holds.
\end{remark}

The composition of correspondences is functorial by the assertion (1) of the following.

\begin{proposition}[{cf.~\cite[Proposition~16.1.2]{Fulton98}}]
\label{prop:functorial}
Let $f \colon X\vdash Y$ and $g \colon Y\vdash Z$ be correspondences.
Then the following statements hold.
\begin{itemize}
\item[(1)] $(g\circ f)_* = g_* \circ f_*$ and $(g\circ f)^* = f^* \circ g^*$.

\item[(2)] $(f^\sT)_* = f^*$ and $(f^\sT)^* = f_*$.

\item[(3)] If $f$ is the graph $\Gamma_\pi$ of a proper morphism $\pi \colon X \to Y$, then $f_* = \pi_*$, the proper pushforward; if $\pi$ is flat, then $f^* = \pi^*$, the flat pullback.
\end{itemize}
\end{proposition}

The following Lefschetz trace formula is particularly useful in studying correspondences.
The formula also indicates that the traces of correspondences acting on cohomology groups are actually numerical invariants provided the standard conjecture \hyperref[conj:C]{$C$} holds.

\begin{proposition}[{Lefschetz trace formula, cf.~\cite[Proposition~1.3.6]{Kleiman68}}]
\label{prop:Lefschetz-trace-formula}
Let $f\colon X\vdash X$, $g \colon X\vdash Y$, and $h \colon Y\vdash X$ be correspondences.
Then we have
\begin{align*}
g \cdot h^\sT &= \sum_{i=0}^{2n} (-1)^i \Tr((h\circ g)^*|_{H^i(X)}), \\
f \cdot \Delta_X &= \sum_{i=0}^{2n} (-1)^i \Tr(f^*|_{H^i(X)}), \\
\cl_{X\times X}(f) \cup \pi_{2n-i} &= (-1)^i \Tr(f^*|_{H^i(X)}).
\end{align*}
Suppose moreover that \cref{conj:C} holds.
Then the last formula reads as
\begin{equation*}
%\label{eq:Lefschetz-trace-formula}
f \cdot \Delta_{2n-i} = (-1)^i \Tr(f^*|_{H^i(X)}).
\end{equation*}
\end{proposition}

The above Lefschetz trace formula is also valid for general homological correspondences.
In this paper, we only focus on algebraic correspondences in $\Corr^0_{\sim}(X,X)$ and refer the reader to \cite[Proposition~1.3.6]{Kleiman68} for the general case.

We also need a useful lemma due to Lieberman.

\begin{lemma}[{Lieberman's Lemma, cf.~\cite[Lemma~2.1.3]{MNP13}}]
\label{lemma:Lieberman}
Let $\phi \colon X\vdash Y$, $\psi \colon X'\vdash Y'$, $f \colon X\vdash X'$, and $g \colon Y\vdash Y'$ be correspondences shown as follows:
\[
\xymatrix{
X \ar@{|-}[r]^{f} \ar@{|-}[d]_{\phi} & X' \ar@{|-}[d]_{\psi} \\
Y \ar@{|-}[r]^{g} & Y'. }
\]
Then we have
\begin{align*}
(\phi \times \psi)_*(f) &= \psi\circ f\circ \phi^\sT, \\
(\phi \times \psi)^*(g) &= \psi^\sT\circ g\circ \phi,
\end{align*}
where we think of $\phi \times \psi \in \Z^{2n}(X\times Y\times X'\times Y')_\bQ/\!\sim$ as a correspondence from $X\times X'$ to $Y\times Y'$ by interchanging the second and the third factors.
In particular, we have the following pullback/pushforward identities on cycle class groups and cohomology groups:
\begin{align*}
((\phi \times \psi)_*(f))_* &= \psi_*\circ f_*\circ \phi^*, \quad ((\phi \times \psi)^*(g))_* = \psi^*\circ g_*\circ \phi_*, \\
((\phi \times \psi)_*(f))^* &= \phi_*\circ f^*\circ \psi^*, \quad ((\phi \times \psi)^*(g))^* = \phi^*\circ g^*\circ \psi_*.
\end{align*}
\end{lemma}

\subsection{Numerical invariants of correspondences}
\label{subsec:inv-corr}

Let $f\in \Z^n(X\times X)_\bQ$ be a correspondence of $X$.
Let $H_{X\times X} \coloneqq \pr_1^*H_X + \pr_2^*H_X$ be a fixed ample divisor on $X\times X$.
The {\it degree} $\deg(f)$ of $f$ (as an algebraic $n$-cycle on $X\times X$) is
\begin{equation}
\label{eq:degree-corr}
\deg(f) \coloneqq f \cdot H_{X\times X}^n = \sum_{i=0}^n \binom{n}{i} f \cdot \pr_1^*H_X^{n-i}\cdot \pr_2^*H_X^{i}.
\end{equation}
We then call
\begin{equation}
\label{eq:degree-k-corr}
\deg_i(f) \coloneqq f \cdot \pr_1^*H_X^{n-i}\cdot \pr_2^*H_X^{i} = f^*H_X^i \cdot H_X^{n-i}
\end{equation}
the {\it $i$-th degree} of $f$ (with respect to $H_X$).
It is also clear that $\deg_i(f)\ge 0$ for any effective correspondence $f$ of $X$.
Note that the two commonly used degrees of $f$ in the literature are $\deg_0(f)$ and $\deg_n(f)$ which are particularly helpful when studying correspondences of curves.

We can also define norms on $\N^n(X\times X)_\bR$ as in \Cref{subsec:cycle}.
For this, it is more natural to work with real correspondences $f\in \Z^n(X\times X)_\bR$:
\begin{equation}
\label{eq:corr-norm}
\|f\| \coloneqq \inf \big\{\deg(f^+ + f^-) : f \equiv f^+ - f^-, \, f^\pm \in \Eff^n(X\times X) \big\},
\end{equation}
and similarly for the dual norm $\|f\|'$ of $f$ as follows:
\begin{equation}
\label{eq:corr-dual-norm}
\|f\|' \coloneqq \sup \big\{|f \cdot g| : g \in \N^{n}(X\times X)_\bR, \|g\| = 1 \big\}.
\end{equation}
Note that for effective correspondences $f$, we always have $\|f\| = \deg(f)$; see \cref{rmk:norm-effective}.

We now introduce norms on the pullback $f^*$ of correspondences $f$ on $\Nk$.
First, we have the following natural operator norms on $\End_\bR(\Nk)$ induced from the norm $\norm{\cdot}$ and the dual norm $\norm{\cdot}'$ on $\Nk$ defined in \Cref{subsec:cycle}:
for any $\phi \in \End_\bR(\Nk)$,
\begin{equation}
\label{eq:corr-operator-norm}
\|\phi\| \coloneqq \sup \big\{\|\phi(\alpha)\| : \alpha\in \Nk, \|\alpha\|=1 \big\},
\end{equation}
and
\begin{align}
\label{eq:corr-operator-norm'}
\begin{split}
\|\phi\|' &\coloneqq \sup \big\{\|\phi(\alpha )\|' : \alpha \in \Nk, \|\alpha\|' = 1 \big\} \\
&= \sup \big\{|\phi(\alpha ) \cdot \beta| : \alpha \in \Nk, \beta \in \N^{n-k}(X)_\bR, \|\alpha\|' = \|\beta\| = 1 \big\}.
\end{split}
\end{align}

Secondly, for a fixed $k$ with $0\le k\le n$, we consider the following group homomorphism
\[
\begin{array}{ccc}
\Phi_k \colon \Z^n(X\times X)_\bR & \lra & \End_\bR(\Nk) \\
f & \mapsto & f^*|_{\Nk},
\end{array}
\]
which naturally factors through $\N^n(X\times X)_\bR$ by definition.
Also, $\Phi_k$ is surjective since
\[
\N^n(X\times X)_\bR \supset \Nk \otimes \N^{n-k}(X)_\bR \isom \End_\bR(\Nk).
\]
We note that for any $f\in \Ker(\Phi_k)$, its $k$-th degree $\deg_k(f)$ is zero as $f^*H_X^k \num 0$.
Now, in a similar vein, we define an invariant $\norm{\cdot}_1$ on $\End_\bR(\Nk)$ as follows: for any (real) correspondence $f$ of $X$,
\begin{equation}
\label{eq:norm-1}
\big\|{f^*|_{\Nk}}\big\|_1 \coloneqq \inf\left\{ 
\deg_k(f^+ + f^-) \ :
\begin{array}{l}
\Phi_k(f - f^+ + f^-) = 0, \\
f^\pm \in \Eff^n(X\times X)
\end{array}
\right\}.
\end{equation}

\begin{lemma}
\label{lemma:norm}
The above assignment \eqref{eq:norm-1} gives a norm on the $\bR$-vector space $\End_\bR(\Nk)$.
\end{lemma}
\begin{proof}
First, it is easy to check that this assignment is nonnegative, homogeneous, and satisfies the triangle inequality.

We only need to check that for any real correspondence $f$ of $X$ this assignment attains zero only if $f^*|_{\Nk}$ is the zero map, i.e., $f\in \Ker(\Phi_k)$.
To this end, by the definition of the pullback action of correspondences on $\Nk$, it suffices to show that if $\|f^*|_{\Nk}\|_1=0$, then $f \cdot (Z_1 \times Z_2) = f \cdot \pr_1^*Z_1 \cdot \pr_2^*Z_2 = Z_1 \cdot f^*Z_2 = 0$ for any effective cycles $Z_1$ and $Z_2$ on $X$ of codimension $n-k$ and $k$, respectively.
Write $W = Z_1\times Z_2$.
For any decomposition $f = f^+ - f^- + g$, where $f^\pm\in \Z^n(X\times X)_\bR$ are effective and $g\in \Ker(\Phi_k)$, we have
\[
|f \cdot W| \le |f^+ \cdot W| + |f^- \cdot W| + |g \cdot W| = |f^+ \cdot W + |f^- \cdot W|.
\]
Now, applying \cref{lemma:boundedness} (proved independently later) to $f^\pm$ and $W = Z_1\times Z_2$, we get that
\[
|f^+ \cdot W| + |f^- \cdot W| \le C \deg_{n-k}(W) \deg_k(f^+ + f^-).
\]
Taking infimum over all decompositions of $f$, we see that $|f\cdot W| = 0$ and hence the lemma follows.
\end{proof}

\begin{remark}
\label{rmk:norm-eff-corr}
Note that for any effective correspondence $f$ of $X$, we have
\[
\big\|{f^*|_{\Nk}}\big\|_1 = \deg_k(f).
\]
First, it is easy to see that $\|f^*|_{\Nk}\|_1 \le \deg_k(f)$.
For any effective correspondences $f^\pm$ such that $\Phi_k(f - f^+ + f^-) = 0$, or equivalently, $f - f^+ + f^- = g \in \Ker(\Phi_k)$, we have
\[
\deg_k(f^+ + f^-) \ge \deg_k(f^+) = \deg_k(f^+ + g) = \deg_k(f + f^-) \ge \deg_k(f).
\]
Taking infimum, we get that $\|f^*|_{\Nk}\|_1 \ge \deg_k(f)$, hence equalities.
\end{remark}

Likewise, we may endow the cohomology groups $H^i(X)$ with the dual norm $\norm{\cdot}'_{\iota}$ (with respect to the field isomorphism $\iota\colon \ol\bQ_\ell \isom \bC$): for any $\alpha\in H^i(X)$,
\begin{equation}
\label{eq:dual-norm-coh}
\|\alpha\|'_{\iota} \coloneqq \sup \big\{ |\alpha \cup \beta|_{\iota} : \beta\in H^{2n-i}(X), \, \norm{\beta}_{\iota} = 1 \big\}.
\end{equation}
Then the induced operator norm of the pullback $f^*$ of a correspondence $f$ on $H^i(X)$ is given by
\begin{align}
\label{eq:operator-dual-norm-coh}
\begin{split}
\big\|f^*|_{H^{i}(X)}\big\|'_{\iota} &\coloneqq \sup \big\{ \|f^*\alpha\|'_{\iota} : \alpha\in H^{i}(X), \, \|\alpha\|'_{\iota} = 1 \big\} \\
&=\sup \left\{
|f^*\alpha \cup \beta|_{\iota} \ :
\begin{array}{l}
\alpha\in H^{i}(X), \, \beta\in H^{2n-i}(X), \\ \|\alpha\|'_{\iota} = \|\beta\|_{\iota} = 1
\end{array}
\right\}.
\end{split}
\end{align}

\subsection{Boundedness of self-intersections of correspondences}
\label{subsec:boundedness}

In this subsection, we prove a lemma bounding the intersection of two effective correspondences $f$, $g$ of $X$ in terms of all $i$-th degrees of $f$ and $g$ (with respect to a fixed ample divisor $H_X$).
The strict intersection analog has been proved in \cite[Lemma~4.1]{Truong20}. 
Here we adopt a similar idea to the case of the usual intersection, which is easier to deal with.

\begin{lemma}
\label{lemma:boundedness}
Let $X$ be a smooth projective variety of dimension $n$ over $\bk$. Then there exists a constant $C>0$ such that for any two effective correspondences $f$, $g$ of $X$, we have
\[
|f \cdot g| \le C \sum_{i=0}^{n} \deg_i(f)\deg_{n-i}(g).
\]
\end{lemma}
\begin{proof}
Let $f$, $g$ be effective correspondences of $X$.
By reduction to the diagonal, the intersection number $f\cdot g$ taken over $X\times X$ is equal to $(f\times g) \cdot \Delta_{X\times X}$, where $\Delta_{X\times X} \coloneqq\{(x_1,x_2,x_1,x_2) : x_1,x_2 \in X\} \subset X\times X\times X\times X$ is the diagonal of $X\times X$.
We define an isomorphism $\tau$ by interchanging the second and the third factors of the fourfold product $Y \coloneqq X\times X\times X\times X$.
Now clearly, $\tau(\Delta_{X\times X}) = \Delta_X \times \Delta_X$, where $\Delta_{X} \coloneqq\{(x,x) : x \in X\} \subset X\times X$ is the diagonal of $X$.
Moreover, using the projection formula, we have
\[
(f\times g) \cdot \Delta_{X\times X} = (f\times g) \cdot \tau^*(\Delta_X \times \Delta_X) = \tau_*(f\times g) \cdot (\Delta_X \times \Delta_X).
\]
For $1\le i \neq j\le 4$, let $p_i$ and $p_{ij}$ denote the natural projections from $X\times X\times X\times X$ to the appropriate $X$ and $X\times X$, respectively.
Then clearly, we have $p_{5-k}\circ \tau = p_{k}$ for $k=2$, $3$, $p_k\circ \tau = p_k$ for $k=1$, $4$, $p_{12} \circ \tau = p_{13}$, and $p_{34} \circ \tau = p_{24}$.

We still denote the natural projections from $X\times X$ to $X$ by $\pr_1$ and $\pr_2$, and let $H_{X\times X} = \pr_1^*H_X + \pr_2^*H_X$ denote the ample divisor on $X\times X$.
In particular, we have $p_1 = \pr_1 \circ p_{12}$, $p_2 = \pr_2 \circ p_{12}$, and so on; on $Y$ we have $p_{ij}^*H_{X\times X} = p_i^*H_X + p_j^*H_X$.
\begin{claim}
\label{claim:boundedness}
There exists a constant $C_1>0$ so that the numerical classes of 
\[
p_{12}^*(C_1 H_{X\times X}^n \pm \Delta_X) \cdot p_{34}^*(C_1 H_{X\times X}^n \pm \Delta_X)
\]
are in the nef cone $\Nef^{2n}(Y)$, which is defined as the closed convex cone in $\N^{2n}(Y)_\bR$ dual to the pseudoeffective cone $\PsEff_{2n}(Y)$; see \cite[\S 2.3]{FL17a}.
\end{claim}

\begin{proof}[Proof of \cref{claim:boundedness}]\renewcommand{\qedsymbol}{}
Note that $H_{X\times X}$ is ample on $X\times X$.
It follows from \cite[Lemma~3.14]{FL17a} that $H_{X\times X}^n$ is in the interior of the so-called pliant cone $\PL^n(X\times X) \subset \N^n(X\times X)_\bR$, which is a closed convex cone contained in the nef cone $\Nef^n(X\times X)$ but behaves better.
We refer to \cite[\S 3]{FL17a} for details.
Thus, there is a constant $C_1>0$ such that $C_1 H_{X\times X}^n \pm \Delta_X$ are still in the interior of $\PL^n(X\times X)$.
Since pliant cones are preserved by pullbacks of morphisms and by intersection products (see \cite[Remarks~3.4 and 3.6]{FL17a}), we get that the numerical classes of $p_{12}^*(C_1 H_{X\times X}^n \pm \Delta_X) \cdot p_{34}^*(C_1 H_{X\times X}^n \pm \Delta_X)$ are in the pliant cone $\PL^{2n}(Y)$ of $Y$; in particular, they lie in the nef cone $\Nef^{2n}(Y)$ of $Y$.
We hence prove the claim.
\end{proof}

Using the above \cref{claim:boundedness}, it is not hard to deduce that
\[
C_1^2 \, p_{12}^*H_{X\times X}^n \cdot p_{34}^*H_{X\times X}^n \pm p_{12}^*\Delta_X \cdot p_{34}^*\Delta_X \in \Nef^{2n}(Y).
\]
Now, as $\tau_*(f\times g)$ is effective and nef cones are dual to pseudoeffective cones, we obtain that
\begin{align*}
&\quad \ \ |\tau_*(f\times g) \cdot (\Delta_X \times \Delta_X)| \le C_1^2 \, \tau_*(f\times g) \cdot p_{12}^*H_{X\times X}^n \cdot p_{34}^*H_{X\times X}^n \\
&= C_1^2 \, (f\times g) \cdot \tau^*p_{12}^*H_{X\times X}^n \cdot \tau^*p_{34}^*H_{X\times X}^n = C_1^2 \, (f\times g) \cdot p_{13}^*H_{X\times X}^n \cdot p_{24}^*H_{X\times X}^n \\
&= C_1^2 \ p_{12}^*f \cdot p_{34}^*g \cdot (p_1^*H_X + p_3^*H_X)^n \cdot (p_2^*H_X + p_4^*H_X)^n \\
&\le C_1^2 \, 2^{2n} \sum_{0\le i, \, j\le n} (p_{12}^*f \cdot p_1^*H_X^{n-i} \cdot p_2^*H_X^{j}) \cdot (p_{34}^*g \cdot p_3^*H_X^{i} \cdot p_4^*H_X^{n-j}) \\
&= C_1^2 \, 2^{2n} \sum_{0\le i, \, j\le n} (f \cdot \pr_1^*H_X^{n-i} \cdot \pr_2^*H_X^{j}) \cdot (g \cdot \pr_1^*H_X^{i} \cdot \pr_2^*H_X^{n-j}) \\
&= C_1^2 \, 2^{2n} \sum_{0\le i \le n} (f \cdot \pr_1^*H_X^{n-i} \cdot \pr_2^*H_X^{i}) \cdot (g \cdot \pr_1^*H_X^{i} \cdot \pr_2^*H_X^{n-i}) \\
&= C_1^2 \, 2^{2n} \sum_{0\le i \le n} \deg_i(f)\deg_{n-i}(g),
\end{align*}
where the second and the fifth lines follow from the projection formula, and the second last line follows from the fact that in order to make the intersection meaningful one has to require $n-i+j\le n$ and $n-j+i\le n$, i.e., $i=j$.
We hence prove \cref{lemma:boundedness}.
\end{proof}

\begin{remark}
\begin{enumerate}[(1)]
\item In the case $\bk = \bC$, a different proof, for the strict intersection of positive closed currents - using regularization of positive closed currents - was given in \cite{DN11} for smooth complex projective varieties and \cite{DNT12} for compact K\"ahler manifolds.

\item When $n=1$ and $f=g$ is an effective $(a,b)$-correspondence of the curve $X$, the above inequality reads as $|f\cdot f| \le 2Cab$.
On the other hand, by the Castelnuovo--Severi inequality (see, e.g., \cite[Example~16.1.10]{Fulton98}), we know that for any $(a,b)$-correspondence $f$ of a curve, $f\cdot f \le 2ab$.
Whether the self-intersection $f\cdot f$ has a uniform lower bound is called the Bounded Negativity Conjecture (see, e.g., \cite{BNC13}).
So \cref{lemma:boundedness} actually provides certain boundedness for negative self-intersections of effective divisors on surfaces of product type, which might be of its independent interest.
\end{enumerate}
\end{remark}

\section{Dynamical correspondences and effective finite correspondences}
\label{section:dyn-corr}

In this section, we first introduce the notion of dynamical correspondences initially studied by Dinh and Sibony \cite{DS06b}.
We collect some basic materials on dynamical correspondences, including their definition, dynamical compositions, dynamical pullbacks under dominant rational maps, and a fundamental result concerning the numerical dynamical degrees $\lambda_k$ of dynamical correspondences.
We then discuss dynamical aspects of effective finite correspondences, which already have their own significance in constructions of various triangulated categories of (effective) motives over $\bk$.

\subsection{Dynamical correspondences}
\label{subsec:dyn-corr}

Dynamical correspondences are natural generalizations of dominant rational maps that have been extensively studied for decades.
In this subsection, we discuss their basic properties.
We refer to \cite{DS06b,DS08,Truong20} for details.

\begin{definition}[Dynamical correspondence]
\label{def:dyn-corr}
Let $X$ and $Y$ be smooth projective varieties of dimension $n$ over $\bk$.
A correspondence $f\in \Z^n(X\times Y)_\bQ$ is {\it dominant}, if for each irreducible component $f_i$ of $f$, the natural restriction maps $\pr_1|_{f_i} \colon f_i \to X$ and $\pr_2|_{f_i} \colon f_i \to Y$ induced from the projections $\pr_1 \colon X\times Y \to X$ and $\pr_2 \colon X\times Y \to Y$, respectively, are both surjective (and hence generically finite).
We say that $f$ is a {\it dynamical correspondence}, if $f$ is both effective and dominant.
\end{definition}

Clearly, the graph $\Gamma_\pi$ of a dominant rational map $\pi\colon X\ratmap Y$ is a dynamical correspondence.
We can compose dynamical correspondences just like how we compose dominant rational maps (see \cite[\S 3.2]{DS06b}, \cite[\S 3]{DS08}, or \cite[\S 3.1]{Truong20}).

\begin{definition}[Dynamical composition]
\label{def:dyn-comp}
Let $f \in \Z^n(X\times Y)_\bQ$ and $g\in \Z^n(Y\times Z)_\bQ$ be two dynamical correspondences.
Assume that they are both irreducible for simplicity since the general case can be defined by linearity.
By generic flatness (see, e.g., \cite[\href{https://stacks.math.columbia.edu/tag/052A}{Proposition~052A}]{stacks-project}), there are nonempty Zariski open subsets $U\subseteq X$ and $V\subseteq Y$ such that the restriction maps $\pr_1|_f \colon f \to X$ and $\pr_1|_g \colon g \to Y$ are finite and flat over $U$ and $V$, respectively.
By shrinking $U$ (if necessary), we may assume that the strict image
\[
f(U) \coloneqq \pr_2(f \cap \pr_1^{-1}(U))
\]
of $U$ under $f$ is contained in $V$.
It follows that for any point $x\in U$, the strict image $f(x)$ of $x$ under $f$ is a finite subset of $V$, whose strict image under $g$ is still finite.
The {\it dynamical composition} $g\dyncomp f$ is then defined as the Zariski closure
\[
\ol{\{(x, g(f(x))) \in U \times Z : x\in U\}}
\]
of the composite graph in $X\times Z$.
Alternatively, let $f_0$ (resp. $g_0$) be a nonempty Zariski open subset of $f$ (resp. $g$) which is finite over some open $U\subseteq X$ (resp. open $V\subseteq Y$).
Then $g\dyncomp f$ is the pushforward of the scheme-theoretic closure of the scheme-theoretic intersection $(f_0\times Z)\cap (U\times g_0) \subset U\times V\times Z$ in $X\times Y\times Z$ under $\pr_{13}\colon X\times Y\times Z\to X\times Z$.
One can check that the resulting correspondence is a dynamical correspondence, independent of the choices of $f_0$ and $g_0$.
\end{definition} 

Intuitively, $g\dyncomp f$ is the same as $\pr_{13,*}((f \times Z) \cap (X\times g))$ with components of dimension greater than $n$ and components whose projections to the factors of $X\times Z$ are not surjective removed. Here we use $\cap $ for the scheme-theoretic intersection.

For a dynamical correspondence $f$ of $X$ and any $m\in \bN_{>0}$, we denote by $f^{\Diamond m}$ the {\it $m$-th dynamical iterate} of $f$, which is still a dynamical correspondence of $X$.
Unlike the usual composition defined in \Cref{subsec:comp}, this dynamical composition is, in general, not functorial in the sense that $(g\dyncomp f)^* \neq f^* \circ g^*$. For example, let $f([x:y:z])=[yz:xz:xy]$ be the standard Cremona map of $X=\mathbb{P}_\bk^2$. Then $f^*$ acts on $H^2(X)$ as multiplication by $2$, while $f^{\Diamond 2}=\id_X$. Hence $\id = (f^{\Diamond 2})^* \neq (f^*)^2 = \times 2^2$ on $H^2(X)$.

\begin{remark}
Compositions of correspondences are functorial but do not necessarily preserve effectiveness.
Dynamical compositions of dynamical correspondences preserve effectiveness but may not be functorial, which leads to a notion called algebraic stability (see below).
The non-functorial nature of dynamical compositions or iterates makes the computation of dynamical degrees very hard.
For instance, the following natural question has been well-known for decades: are dynamical degrees always algebraic numbers?
Recently, a negative answer to this question is given by Bell--Diller--Jonsson \cite{BDJ20}.
\end{remark}

\begin{definition}[Algebraic stability]
\label{def:alg-stable}
Let $f\in \Z^n(X\times X)_\bQ$ be a dynamical correspondence of $X$.
Fix an integer $0\le k\le n$.
We say that $f$ is {\it algebraically $k$-stable}, if for all $m\in \bN_{>0}$,
\[
(f^{\Diamond m})^*|_{H^{2k}(X)} = (f^*)^m|_{H^{2k}(X)} = (f^{\circ m})^*|_{H^{2k}(X)}.
\]
It is called {\it algebraically stable}, if for all $m\in \bN_{>0}$,
\[
(f^{\Diamond m})^*|_{H^{\bullet}(X)} = (f^*)^m|_{H^{\bullet}(X)} = (f^{\circ m})^*|_{H^{\bullet}(X)}.
\]
\end{definition}

Clearly, if $f$ is (the graph of) a self-morphism of $X$, then it is algebraically $k$-stable for all $k$.
The study of algebraic stability has attracted a lot of attention \cite{DF01,Truong13,DL16,DF16}.
Here we collect some examples.

\begin{example}
\begin{enumerate}[(1)]
\item Let $f = \sum_i a_i \Gamma_i$ be a positive combination of the graphs $\Gamma_i$ of self-morphisms of $X$ with $a_i\in \bQ_{>0}$.
Then it is also algebraically $k$-stable for all $k$.
Indeed, one can check that the $m$-th dynamical iterates $f^{\Diamond m}$ are the same as the $m$-th iterates $f^{\circ m}$.

\item Another interesting example is due to Voisin \cite{Voisin04} and studied by Amerik \cite{Amerik09} later.
For a smooth cubic fourfold $V \subset \bbP^5_\bC$, let $X = F(V)$ be the variety of lines on $V$.
Then $X\subset \mathbb{G}(1,5)$ is known to be a hyper-K\"ahler $4$-fold.
There is a rational self-map $f$ of $X$ defined by sending a general line $l\subset V$ to the line $l'$, where $l'\cup l = P_l \cap V$ and $P_l$ is the unique plane in $\bbP^5_\bC$ tangent to $V$ along $l$.
Amerik showed that the indeterminacy locus of $f$ is a smooth surface and $f$ is algebraically stable.
\end{enumerate}
\end{example}

For the convenience of the reader, we recall a fundamental result on the existence of the numerical dynamical degrees $\lambda_k$ of dynamical correspondences due to the second author \cite{Truong20}.

\begin{theorem}[{cf.~\cite[Theorem~1.1(1)]{Truong20}}]
\label{thm:lambda-k-exists}
Let $X$ be a smooth projective variety of dimension $n$ over $\bk$.
Let $f$ be a dynamical correspondence of $X$.
Then for any $0\le k\le n$, the limit
\[
\lim_{m\to \infty} \big\|{(f^{\Diamond m})^*|_{\Nk}}\big\|^{1/m}
\]
defining the $k$-th numerical dynamical degree $\lambda_k(f)$ of $f$ exists and is equal to
\[
\lim_{m\to \infty} (\deg_k(f^{\Diamond m}))^{1/m} = \lim_{m\to \infty} ((f^{\Diamond m})^*H_X^k \cdot H_X^{n-k})^{1/m},
\]
where $H_X$ is an (arbitrary) ample divisor on $X$.
\end{theorem}

We also recall dynamical pullbacks of dynamical correspondences by dominant rational maps, which will be used in the proof of \cref{lemma:Gr-Kummer}.

\begin{definition}[Dynamical pullback]
\label{def:dyn-pullback}
Let $\pi_i \colon X_i \ratmap Y_i$ be dominant rational maps of smooth projective varieties of dimension $n$ over $\bk$ for $i=1,2$.
Let $g$ be a dynamical correspondence from $Y_1$ to $Y_2$.
Assume that $g$ is also irreducible.
By generic flatness (see, e.g., \cite[\href{https://stacks.math.columbia.edu/tag/052A}{Proposition~052A}]{stacks-project}), there are nonempty Zariski open subsets $U_i\subseteq X_i$ and $V_i\subseteq Y_i$ such that $U_i = \pi_i^{-1}(V_i)$ and $\pi_i|_{U_i} \colon U_i \to V_i$ are finite and flat morphisms.
Then the {\it dynamical pullback} $(\pi_1\times \pi_2)^{\dynpb}(g)$ of $g$ under $(\pi_1,\pi_2)$ (or $\pi_1\times \pi_2$) is defined by the Zariski closure
\[
\ol{(\pi_1\times \pi_2|_{U_1\times U_2})^{-1}(g \cap (V_1\times V_2))}
\]
of the scheme-theoretic inverse image $(\pi_1\times \pi_2|_{U_1\times U_2})^{-1}(g \cap (V_1\times V_2))$ in $X_1\times X_2$.
It is indeed a dynamical correspondence from $X_1$ to $X_2$ (i.e., both of the induced projections to $X_i$ are surjective), by looking at the commutative diagram below:
\[
\begin{tikzcd}
& & {(\pi_1\times \pi_2)^{\dynpb}(g)} \arrow[ld, two heads, "\pr_1"'] \arrow[rd, two heads, "\pr_2"] \arrow[d, two heads, dashed] & & \\
U_1 \arrow[r, hook] \arrow[d, two heads, "\pi_1|_{U_1}"'] & {X_1} \arrow[d, two heads, dashed, "\pi_1"'] & {g} \arrow[ld, two heads, "\pr_1"'] \arrow[rd, two heads, "\pr_2"] & {X_2} \arrow[d, two heads, dashed, "\pi_2"] & U_2 \arrow[l, hook'] \arrow[d, two heads, "\pi_2|_{U_2}"] \\
V_1 \arrow[r, hook] & {Y_1} & & {Y_2} & V_2. \arrow[l, hook']
\end{tikzcd}
\]
The reducible case can be defined linearly.
\end{definition}

\begin{remark}
\label{rmk:dyn-pullback}
Consider the dominant rational map $\pi_1\times \pi_2\colon X_1\times X_2 \ratmap Y_1\times Y_2$.
Let $(\pi_1\times \pi_2)^{-1}(g)$ denote the total transform of $g$ under $\pi_1\times \pi_2$, i.e.,
\[
(\pi_1\times \pi_2)^{-1}(g) \coloneqq \pr_{12}(\Gamma_{\pi_1\times \pi_2} \cap \pr_{34}^{-1}(g)),
\]
where $\pi_{ij}$ denotes the projection from $X_1\times X_2 \times Y_1\times Y_2$ to the product of the $i$-th and $j$-th factors. Clearly, the difference between the total transform $(\pi_1\times \pi_2)^{-1}(g)$ and the dynamical pullback $(\pi_1\times \pi_2)^{\dynpb}(g)$ lies in the indeterminacy locus $I_{\pi_1\times \pi_2}$ and the non-flat locus $NF_{\pi_1\times \pi_2} \subset X_1\times X_2$ of $\pi_1\times \pi_2$.

If we assume further that both $\pi_i\colon X_i \to Y_i$ are morphisms, then one can check that the image of $(\pi_1\times \pi_2)^{\dynpb}(g)$ under $\pi_1\times \pi_2$ is exactly $g$.
Indeed, by definition, one easily has
\[
(\pi_1\times \pi_2)((\pi_1\times \pi_2)^{\dynpb}(g)) = \ol{g \cap (V_1\times V_2)} = g.
\]
It follows from the definition of proper pushforward that
\[
(\pi_1\times \pi_2)_*((\pi_1\times \pi_2)^{\dynpb}(g)) = \deg(\pi_1\times \pi_2) \, g.
\]
\end{remark}

\begin{example}
Let $\pi\colon \bbP_\bC^1 \rightarrow \bbP_\bC^1$ be the morphism defined by $\pi([z_1:z_2]) = [z_1^d,z_2^d]$ (hence of degree $d$).
Let $g$ be the graph of a polynomial map $p\colon \bbP_\bC^1\rightarrow \bbP_\bC^1$.
Then the dynamical pullback $f\coloneqq (\pi\times \pi)^{\dynpb}(g)$ is given by
\[
\{([z_1,z_2],[w_1,w_2])\in \bbP_\bC^1\times \bbP_\bC^1 : [w_1^d:w_2^d] = p([z_1^d:z_2^d])\},
\]
which is a $1$-to-$d$ correspondence (i.e., $\deg_0(f)=d$).
Note that
\[
(\pi\times \pi)^{\dynpb} (g^{\Diamond m})=\{([z_1,z_2],[w_1,w_2])\in \bbP_\bC^1\times \bbP_\bC^1 : [w_1^d:w_2^d] = p^{m}([z_1^d:z_2^d])\}
\]
is a $1$-to-$d$ correspondence and $f^{\Diamond m}$ is a $1$-to-$d^m$ correspondence.
One can check that
\[
f^{\Diamond m} = d^{m-1} (\pi\times \pi)^{\dynpb} (g^{\Diamond m}).
\]
Hence, since the degree of $\pi\times \pi$ is $d^2$, we have $(\pi\times \pi)_*(f^{\Diamond m})=d^{m+1} g^{\Diamond m}$.
\end{example}

In general, we have the following lemma for dynamical pullbacks of dynamical correspondences.

\begin{lemma}
\label{lamma:pushforward-composition-dyn-pullback}
Let $\pi\colon X\to Y$ be a generially finite surjective morphism of smooth projective varieties over $\bk$.
Let $g_1$, $g_2$ be dynamical correspondences of $Y$.
Let $f_i = (\pi\times \pi)^{\dynpb}(g_i)$ be the dynamical pullback of $g_i$ with $i=1,2$.
Then we have
\begin{align*}
\deg(\pi) (\pi\times \pi)^{\dynpb}(g_2\dyncomp g_1) &= f_2\dyncomp f_1, \\
(\pi\times \pi)_*(f_2\dyncomp f_1) &= \deg(\pi)^3 g_2\dyncomp g_1.
\end{align*}
\end{lemma}
\begin{proof}
The second one follows immediately from the first one and \cref{rmk:dyn-pullback}.
For the first one, note that by definition, there is an appropriate Zariski open subset $U_1\subseteq X$ such that for any $x\in U_1$, we have 
\begin{equation*}
\begin{gathered}
(\pi\times \pi)^{\dynpb}(g_2\dyncomp g_1)(x) = \pi^{-1}\circ g_2\circ g_1\circ \pi(x), \quad \text{and} \\
(f_2\dyncomp f_1)(x) = \pi^{-1}\circ g_2\circ \pi \circ \pi^{-1}\circ g_1\circ \pi(x) = \pi^{-1}\circ g_2\circ g_1\circ \pi(x),
\end{gathered}
\end{equation*}
where the left-hand sides denote the strict images of $x$ under correspondences.
Clearly, as correspondences (or rather, effective algebraic $n$-cycles on $X\times X$), we have $f_2\dyncomp f_1 = c \, (\pi\times \pi)^{\dynpb}(g_2\dyncomp g_1)$ for some constant $c>0$.
Comparing $\deg_0$ of both sides, one can see that $c = \deg(\pi)$.
\end{proof}

The next lemma has been implicitly proved by the second author in the more general setting of semi-conjugate correspondences.
The proof relies on a quantitative version of Chow's moving lemma (see, e.g., \cite[Lemma~2.4]{Truong20} or \cite{Roberts72}); as an important consequence, it is deduced that numerical dynamical degrees are birational invariants (see \cite[Lemma~5.5]{Truong20}).

\begin{lemma}[{cf.~\cite[Lemmas~3.4 and 5.4]{Truong20}}]
\label{lemma:deg-k-invariant-under-gen-finite}
Let $\pi\colon X\to Y$ be a generically finite surjective morphism of smooth projective varieties of dimension $n$ over $\bk$.
Let $g$ be a dynamical correspondence of $Y$, and $f\coloneqq (\pi\times \pi)^{\dynpb}(g)$ the dynamical pullback of $g$ under $\pi\times \pi$.
Then for any $0\le k\le n$, we have $\deg_k(f) \sim \deg_k(g)$, i.e., there is a constant $C>0$, independent of $g$ and $f$, such that
\[
C^{-1} \deg_k(g) \le \deg_k(f) \le C \deg_k(g).
\]
\end{lemma}

\begin{remark}
We note that in the above lemma, if $g$ is merely an effective correspondence of $Y$ but such that $f \coloneqq (\pi\times \pi)^*(g)$, the pullback of $g$ by $\pi\times \pi$ (viewed as a correspondence from $X\times X$ to $Y\times Y$), is also effective, then one can easily get that $\deg_k(f) \sim \deg_k(g)$.
In fact, by \cref{rmk:norm-eff-corr}, the $k$-th degree of $g$ (resp. $f$) is equal to the norm $\|g^*|_{\N^k(Y)_\bR}\|_1$ (resp. $\|f^*|_{\N^k(X)_\bR}\|_1$).
Also, the proper pushforward $\pi_* \colon \N^k(X)_\bR \to \N^k(Y)_\bR$ is surjectve because $\pi$ is generically finite surjective.
It follows that
\[
\N^k(X)_\bR \isom \N^k(Y)_\bR \oplus \Ker(\pi_*).
\]
Using this decomposition and the dual operator norm \eqref{eq:corr-operator-norm'}, one can see that the two dual operator norms $\|g^*|_{\N^k(Y)_\bR}\|'$ and $\|f^*|_{\N^k(X)_\bR}\|'$ are equivalent, so are $\deg_k(g)$ and $\deg_k(f)$.
\end{remark}

\subsection{Effective finite correspondences}
\label{subsec:finite-corr}

As discussed in \Cref{subsec:dyn-corr}, typical examples of dynamical correspondences are graphs of dominant rational maps (and surjective morphisms).
The assumption that dynamical correspondences are dominant is used to assure the existence of a well-defined dynamical composition of two dynamical correspondences given in \cref{def:dyn-comp}.
In the case when correspondences are given by graphs of morphisms, we do not need these morphisms to be surjective to compose them.
The key point is that the projection of the graph of a self-morphism to the first factor has all fibers finite (in fact, it is an isomorphism).
This latter property is shared by so-called finite correspondences.

\begin{definition}[Finite correspondence]
\label{def:finite-corr}
Let $X$, $Y$ be smooth projective varieties of dimension $n$ over $\bk$.
Let $f\in \Z^n(X\times Y)_\bQ$ be a correspondence.
We say that $f$ is a {\it finite correspondence}, if for each irreducible component $f_i$ of $f$, the restriction map $\pr_1|_{f_i}\colon f_i \to X$ induced from the first projection $\pr_1 \colon X\times Y \to X$ is finite and surjective.
\end{definition}

\begin{remark}
\label{rmk:finite-corr-comp}
Finite correspondences were introduced by Suslin and Voevodsky in constructions of various triangulated categories $\mathbf{DM}^{\mathrm{eff}}(\bk)$ of (effective) motives over $\bk$ (see \cite{MVW06}).
Typical examples include graphs of (not necessarily surjective) morphisms and transposes of graphs of finite surjective morphisms (see \cite[Example~1.2]{MVW06}).
An advantage of finite correspondences is that we do not need to mod any equivalence to speak of their compositions.

Precisely, let $f\in \Z^n(X\times Y)_\bQ$ and $g\in \Z^n(Y\times Z)_\bQ$ be two finite correspondences.
Without loss of generality, we may assume that both $f$ and $g$ are irreducible.
Then \cite[Lemma~1.7]{MVW06} asserts that $\pr_{12}^*f = f \times Z$ intersects $\pr_{23}^*g = X \times g$ properly and each component of the $\pr_{13}$-pushforward of $\pr_{12}^*f \cdot \pr_{23}^*g = \pr_{12}^*f \cap \pr_{23}^*g$ is finite and surjective over $X$.
Hence $g\circ f\in \Z^n(X\times Z)_\bQ$ is also a finite correspondence (note, however, that in \Cref{subsec:comp} one has to mod out at least rational equivalence when composing two general correspondences).
In particular, if $f$ and $g$ are effective finite correspondences, so is their composition $g\circ f$.
Therefore, our \cref{thm:B}\eqref{Assertion:B-2} applies to all effective finite correspondences.
\end{remark}

\begin{remark}
\label{rmk:finite-corr-alg-stable}
One can consider dynamical compositions of effective finite correspondences.
Precisely, let $f\in \Z^n(X\times Y)_\bQ$ and $g\in \Z^n(Y\times Z)_\bQ$ be two effective finite correspondences.
Then by \cref{def:dyn-comp} and the above \cref{rmk:finite-corr-comp}, the dynamical composition $g\dyncomp f$ coincides with the composition $g\circ f$, since the restriction maps $\pr_1|_f\colon f \to X$ and $\pr_1|_g\colon g \to Y$ are finite over $X$ and $Y$, respectively.
In particular, all effective finite correspondences are algebraically stable in the sense of \cref{def:alg-stable}, and hence we may speak of their dynamical degrees which are nothing but the spectral radii of the corresponding linear maps.
\end{remark}

\section{Proofs of main results}
\label{section:proof}

\subsection{Some auxiliary results}

The lemma below should be standard.
Using it, many of our comparison problems can be reduced to merely consider the (nontrivial) direction $\ge$.

\begin{lemma}
\label{lemma:easy-direction}
Let $f\in \Z^n(X\times X)_\bQ$ be an arbitrary correspondence of $X$.
Then there is a positive constant $C>0$, independent of $f$, such that for any $0\le k\le n$,
\[
\big\|f^*|_{\Nk}\big\| \le C \, \big\|f^*|_{H^{2k}(X)}\big\|_{\iota}.
\]
In particular, one has
\[
\rho(f^*|_{\Nk}) \le \rho(f^*|_{H^{2k}(X)})_{\iota}.
\]
\end{lemma}
\begin{proof}
Let $Z_1,\ldots,Z_m \in \Z^{k}(X)_\bQ$ and $Z'_1,\ldots,Z'_m \in \Z^{n-k}(X)_\bQ$ be such that their numerical classes form $\bR$-bases of $\N^{k}(X)_\bR$ and $\N^{n-k}(X)_\bR$, respectively.
Then it is easy to see that
\[
\max_{1\le i, \, j \le m} |f^*Z_i \cdot Z'_j|
\]
gives rise to a norm on $f^*|_{\Nk}$.
Similarly, let $\alpha_1,\ldots,\alpha_{b_{2k}}$ and $\alpha'_1,\ldots,\alpha'_{b_{2k}}$ be $\bQ_\ell$-bases of $H^{2k}(X)$ and $H^{2n-2k}(X)$, respectively.
Then we have a norm of $f^*|_{H^{2k}(X)}$ as follows:
\[
\big\|f^*|_{H^{2k}(X)}\big\|_{\iota} \coloneqq \max_{1\le r, \, s \le b_{2k}} |f^*\alpha_r \cup \alpha'_s|_{\iota}.
\]
Write $\cl_X(Z_i) = \sum_r a_{i,r} \alpha_r$ and $\cl_X(Z'_j) = \sum_s a'_{j,s} \alpha'_s$ for some $a_{i,r},a'_{j,s} \in \bQ_\ell$.
It follows from the compatibility of the intersection and the cup product under the cycle class map that
\[
|f^*Z_i \cdot Z'_j| \le \sum_{1\le r, \, s \le b_{2k}} |a_{i,r} a'_{j,s} f^*\alpha_r \cup \alpha'_s|_{\iota} \le C \max_{1\le r, \, s \le b_{2k}} |f^*\alpha_r \cup \alpha'_s|_{\iota},
\]
where $C = b_{2k}^2 \max_{i,j,r,s}|a_{i,r} a'_{j,s}|_{\iota}$, independent of $f$.
We thus prove the first part of \cref{lemma:easy-direction}.
The second part follows from the first part applied to $f^{\circ m}$ and the spectral radius formula.
\end{proof}

We first verify that NC \cref{conj:NC} implies DDC \cref{conj:DDC}.
In fact, we prove a slightly stronger form as follows.

\begin{lemma}
\label{lemma:norm-DDC}
Suppose that \cref{conj:NC} holds on $X$.
Then for any dynamical correspondence $f$ of $X$, we have that for any $0\le k\le n$,
\begin{equation}
\label{eq:chi-limit}
\chi_{2k}(f)_{\iota} = \lim_{m\to\infty} \big\|(f^{\Diamond m})^*|_{H^{2k}(X)} \big\|_{\iota}^{1/m} = \lambda_k(f),
\end{equation}
and for any $0\le k\le n-1$,
\begin{equation}
\label{eq:Dinh-ineq}
\chi_{2k+1}(f)_{\iota} \coloneqq \limsup_{m\to\infty} \big\|(f^{\Diamond m})^*|_{H^{2k+1}(X)} \big\|_{\iota}^{1/m} \le \sqrt{\lambda_k(f)\lambda_{k+1}(f)}.
\end{equation}
In particular, \cref{conj:DDC} holds.
\end{lemma}
\begin{proof}
Let $f$ be a dynamical correspondence of $X$.
Note that for all $m\in \bN_{>0}$, the dynamical iterates $f^{\Diamond m}$ are still dynamical correspondences by definition.
So by applying the inequality \eqref{eq:Dinh-even} in \cref{conj:NC} to $f^{\Diamond m}$, we have that for any $m\in\bN_{>0}$ and any $0\le k\le n$,
\[
\big\|(f^{\Diamond m})^*|_{H^{2k}(X)}\big\|_{\iota} \le C \, \big\|(f^{\Diamond m})^*|_{\Nk}\big\|,
\]
for some constant $C>0$ independent of $f$ and $m$.
Taking the limsups of the $m$-th roots of both sides yields that
\begin{align*}
\chi_{2k}(f)_{\iota} \coloneqq \limsup_{m\to\infty} \big\|(f^{\Diamond m})^*|_{H^{2k}(X)}\big\|_{\iota}^{1/m} \le \lim_{m\to\infty} \big\|(f^{\Diamond m})^*|_{\Nk}\big\|^{1/m} = \lambda_k(f),
\end{align*}
where the existence of the limit defining $\lambda_k(f)$ follows from \cref{thm:lambda-k-exists}.

On the other hand, by \cref{lemma:easy-direction}, we also have
\begin{align*}
\liminf_{m\to\infty} \big\|(f^{\Diamond m})^*|_{H^{2k}(X)}\big\|_{\iota}^{1/m} \ge \lim_{m\to\infty} \big\|(f^{\Diamond m})^*|_{\N^{k}(X)_\bR}\big\|^{1/m} = \lambda_k(f).
\end{align*}
Combining the above together, we obtain that
\[
\chi_{2k}(f)_{\iota} = \lim_{m\to\infty} \big\|(f^{\Diamond m})^*|_{H^{2k}(X)} \big\|_{\iota}^{1/m} = \lambda_k(f).
\]

Similarly, the inequality \eqref{eq:Dinh-ineq} follows from the inequality \eqref{eq:Dinh-odd} in \cref{conj:NC} applied to $f^{\Diamond m}$ and \cref{thm:lambda-k-exists}.
\end{proof}

\begin{remark}
Assuming \cref{conj:NC}, we see that the $\limsup$ \eqref{eq:chi} defining $\chi_i(f)_{\iota}$ turns out to be a limit for even $i$; \cref{lemma:norm-DDC} also gives further information concerning $\chi_{i}(f)_{\iota}$ for odd $i$.
The latter in turn answers affirmatively \cite[Question~4]{Truong}, where the above inequality \eqref{eq:Dinh-ineq} is named Dinh's inequality.
\end{remark}

We then show that the generalized Weil's Riemann hypothesis (for all varieties) follows from our DDC \cref{conj:DDC} (for all varieties).

\begin{lemma}
\label{lemma:DDC-Serre}
\cref{conj:DDC} holds on $X$ and $X\times X$ implies that \cref{conj:Serre} holds on $X$.
\end{lemma}
\begin{proof}
Let $f$ be a polarized endomorphism of $X$, i.e., $f^*H_X \sim_{\rat} qH_X$ for an ample divisor $H_X$ and a positive integer $q$.
In particular, we have $\lambda_k(f) = q^k$ (see \cref{rmk:lambda-chi}).
To prove \cref{conj:Serre}, it suffices to show that all eigenvalues of $f^*|_{H^i(X,\bQ_\ell)}$ have absolute value $q^{i/2}$ for any field isomorphism $\iota\colon \ol\bQ_\ell \isom \bC$ (in fact, once this has been proved, all eigenvalues of $f^*|_{H^i(X,\bQ_\ell)}$ are algebraic).
It is then equivalent to proving that $\chi_i(f)_{\iota} \le q^{i/2}$ for any $\iota$ by Poincar\'e duality.
For even indices $i=2k$, this is exactly the assertion in \cref{conj:DDC} applied to $X$ and $f$.
For odd indices $i=2k+1$, one just considers the induced product morphism $f\times f \colon X\times X \to X\times X$, which is still a polarized endomorphism such that $(f\times f)^*H_{X\times X} \sim_{\rat} q H_{X\times X}$ with $H_{X\times X} \coloneqq \pr_1^*H_X + \pr_2^*H_X$.
Applying \cref{conj:DDC} to $X\times X$ and $f\times f$ yields that $\chi_{4k+2}(f\times f)_{\iota} = \lambda_{2k+1}(f\times f) = q^{2k+1}$.
On the other hand, the K\"unneth formula asserts that $\chi_{2k+1}(f)_{\iota}^2 \le \chi_{4k+2}(f\times f)_{\iota}$.
It follows that $\chi_{2k+1}(f)_{\iota} \le q^{(2k+1)/2}$, as desired.
\end{proof}

The lemma below shows that the generalized semisimplicity conjecture (for polarized endomorphisms) follows from our NC \cref{conj:NC}.

\begin{lemma}
\label{lemma:norm-SS}
\cref{conj:NC} implies \cref{conj:SS}.
\end{lemma}
\begin{proof}
Let $f$ be a polarized endomorphism of $X$, i.e., $f^*H_X \sim_{\rat} qH_X$ for an ample divisor $H_X$ and a positive integer $q$.
Then one has $(f^m)^*H_X \sim_{\rat} q^m H_X$ for any $m\in\bN_{>0}$.
In particular, the $k$-th degree $\deg_k(f^m)$ of $f^m$ with respect to $H_X$ is equal to $q^{mk}H_X^n$.
Then by \cref{rmk:norm-eff-corr} we have that
\[
\big\|{(f^m)^*|_{\Nk}}\big\| \sim \big\|{(f^m)^*|_{\Nk}}\big\|_1 = \deg_k(f^m) = q^{mk}H_X^n.
\]
Applying \cref{conj:NC} to $f^m$ yields that
\[
\big\|{(f^m)^*|_{H^{2k}(X)}}\big\|_{\iota} \lesssim \big\|{(f^m)^*|_{\Nk}}\big\| \sim q^{mk}, \quad\quad \textrm{and}
\]
\[
\big\|{(f^m)^*|_{H^{2k+1}(X)}}\big\|_{\iota} \lesssim \sqrt{\big\|{(f^m)^*|_{\Nk}}\big\|\big\|{(f^m)^*|_{\N^{k+1}(X)_\bR}}\big\|} \sim q^{m(2k+1)/2}.
\]
Combining them together, we have shown that for each $0\le i\le 2n$,
\[
\big\|{(f^*)^m|_{H^{i}(X)}}\big\|_{\iota} = \big\|{(f^m)^*|_{H^{i}(X)}}\big\|_{\iota} \lesssim q^{mi/2}.
\]
It follows that the Jordan canonical form of $f^*|_{H^{i}(X)}$ over $\ol\bQ_\ell$ has no Jordan blocks of size $\ge 2$, for otherwise, there would be certain polynomial growth term $m^{n_i}$ of the norm of $(f^*)^m|_{H^{i}(X)}$.
We thus prove that $f^*|_{H^{i}(X)}$ is semisimple.
\end{proof}

Now, let us recall a generalization of Grothendieck's Hodge index theorem on surfaces to higher-dimensional varieties, which is also known as the Khovanskii--Teissier inequality. We refer to \cite[\S 1.6.A]{Lazarsfeld-I} for a proof. Note that the inequality also follows from the Hodge--Riemann bilinear relation for divisors (see, e.g., \cite[\S 2.5]{Hu20a}). 

\begin{theorem}[Khovanskii--Teissier inequality]
\label{thm:KT}
Let $Z$ be a projective variety of dimension $n$ over $\bk$.
Let $\alpha,\beta\in \N^1(Z)_\bR$ be nef Cartier divisor classes on $Z$.
Then for all $1\le i\le n-1$, we have
\[
(\alpha^{n-i} \cdot \beta^{i})^2 \ge (\alpha^{n-i+1} \cdot \beta^{i-1})(\alpha^{n-i-1} \cdot \beta^{i+1}).
\]
\end{theorem}

From this one can easily deduce, as in the next lemma, that the degree sequences $\{\deg_i(f)\}_i$ of irreducible correspondences $f$ are log-concave.
Here a finite sequence $\{a_i\}_{i=0}^n$ of positive real numbers is said to be {\it log-concave}, if $a_i^2\ge a_{i-1}a_{i+1}$ for all $1\le i\le n-1$.
In particular, there are integers $0\le s\le t\le n$ such that
\[
0 < a_0 < \cdots < a_s = \cdots = a_t > \cdots > a_n > 0.
\]
For meromorphic maps of compact K\"ahler manifolds, this log-concavity property is well-known in Complex Dynamics. The version for irreducible correspondences (possibly on a singular variety) can be found, e.g., in \cite{Truong20}. In this paper, we only treat the smooth case.

\begin{lemma}
\label{lemma:log-concave}
Let $H_X$ be a fixed ample divisor on $X$.
Let $f$ be an irreducible correspondence of $X$.
Then the degree sequence $\{\deg_i(f)\}_i$ is log-concave, where
\[
\deg_i(f) \coloneqq f^*H_X^i\cdot H_X^{n-i} = f\cdot \pr_1^*H_X^{n-i} \cdot \pr_2^*H_X^{i}.
\]
Precisely, for all $1\le i\le n-1$, we have
\[
\deg_i(f)^2 \ge \deg_{i-1}(f)\deg_{i+1}(f).
\]
\end{lemma}
\begin{proof}
We let $\alpha = (\pr_1^*H_X)|_f$ (resp. $\beta = (\pr_2^*H_X)|_f$) be the restriction of the nef divisor $\pr_1^*H_X$ (resp. $\pr_2^*H_X$) to $f$, respectively.
Note that $f\cdot \pr_1^*H_X^{n-i} \cdot \pr_2^*H_X^{i} = \alpha^{n-i} \cdot \beta^{i}$.
Then the log-concavity of the degree sequence $\{\deg_i(f)\}_i$ of $f$ follows readily from \cref{thm:KT}.
\end{proof}

Using the fact that any log-concave sequence is stable under suitable scaling by definition, we obtain a technical lemma as follows, which is particularly useful in practice by reducing the upper bound from the maximum of a log-concave sequence to a prescribed index $k$.
This is also a motivation for us to introduce \cref{conj:Gr}.

\begin{lemma}
\label{lemma:log-concave-sequence}
Let $\{a_j\}_{j=0}^{n}$ be a log-concave sequence of positive real numbers, i.e., $a_j^2\ge a_{j-1}a_{j+1}$ for all $1\le j\le n-1$.
Let $\{b_i\}_{i=0}^{2n}$ be a finite sequence of nonnegative real numbers such that for any positive rational number $r\in \bQ_{>0}$ and any $0\le i\le 2n$,
\begin{equation}
\label{eq:log-concave-sequence}
r^i b_i \le \max_{0\le j\le n} r^{2j} a_j.
\end{equation}
Then we have for any $0\le k\le n$,
\begin{equation}
\label{eq:Dinh-even-sequence}
b_{2k} \le a_{k},
\end{equation}
and for any $0\le k\le n-1$,
\begin{equation}
\label{eq:Dinh-odd-sequence}
b_{2k+1} \le \sqrt{a_{k}a_{k+1}}.
\end{equation}
\end{lemma}
\begin{proof}
By continuity, \cref{eq:log-concave-sequence} holds true for any positive real number $r\in \bR_{>0}$.
Let us first consider the case when $0\le k\le n-1$.
For this fixed $k$, we choose $r_k^2 \coloneqq a_{k}/a_{k+1} \in \bR_{>0}$.
Clearly, the new finite sequence $\{c_{k,j} \coloneqq r_k^{2j} a_j\}_{j=0}^n$ of positive real numbers is log-concave since so is $\{a_j\}_{j=0}^n$ by assumption.
In particular, one has
\[
\cdots \le \frac{c_{k,k-1}}{c_{k,k}} \le \frac{c_{k,k}}{c_{k,k+1}} = \frac{a_{k}}{r_k^2 a_{k+1}} = 1 \le \frac{c_{k,k+1}}{c_{k,k+2}} \le \cdots .
\]
It follows immediately that $c_{k,k-1} \le c_{k,k} = c_{k,k+1} \ge c_{k,k+2}$.
The log-concavity property infers that the maximum of all $c_{k,j} = r_k^{2j} a_j$ can be achieved at $c_{k,k}$ (or $c_{k,k+1}$).
Substituting $r_k$ into the inequality \eqref{eq:log-concave-sequence}, we obtain that for any $0\le i\le 2n$,
\[
r_k^{i} b_{i} \le \max_{0\le j\le n} r_k^{2j} a_j = r_k^{2k} a_k.
\]
In particular, for $i=2k$, we have the inequality \eqref{eq:Dinh-even-sequence}; if $i=2k+1$, the inequality \eqref{eq:Dinh-odd-sequence} follows.

The remaining case $k=n$ has been covered by the case $k=n-1$. Indeed, it follows from the above proof that $r_{n-1}^{2n}b_{2n} \le \max_{0\le j\le n} r_{n-1}^{2j} a_j = c_{n-1,n-1} = c_{n-1,n} = r_{n-1}^{2n} a_n$.
\end{proof}

With the help of the above \cref{lemma:log-concave-sequence}, we can now prove \cref{thm:B}\eqref{Assertion:B-1} in the case when the dynamical correspondence is irreducible.

\begin{lemma}
\label{lemma:Gr+log-concave}
Suppose that Conjecture~\hyperref[conj:Gr]{$G_r(X)$} holds.
Let $f$ be an irreducible dynamical correspondence of $X$.
Then there is a positive constant $C>0$, independent of $f$, such that for any $0\le k\le n$,
\[
 |\Tr(f^*|_{H^{2k}(X)}) | \le C \deg_k(f),
\]
and for any $0\le k\le n-1$,
\[
|\Tr(f^*|_{H^{2k+1}(X)})| \le C \sqrt{\deg_k(f) \deg_{k+1}(f)}.
\]
\end{lemma}
\begin{proof}
We first note that by the proof of \cite[Theorem~2A11]{Kleiman68} which is due to Lieberman, \cref{conj:C} holds.
Then according to the Lefschetz trace formula (see \cref{prop:Lefschetz-trace-formula}), we have for each $0\le i\le 2n$,
\[
|\Tr((G_r \circ f)^*|_{H^{i}(X)})| = |(G_r \circ f) \cdot \Delta_{2n-i}| \le \|\Delta_{2n-i}\|\|G_r \circ f\|',
\]
where $\|\cdot\|'$ denotes the dual norm \eqref{eq:corr-dual-norm} on $\N^n(X\times X)_\bR$.
Since Conjecture~\hyperref[conj:Gr]{$G_r(X)$} holds, there is a constant $C_1>0$, independent of $f$, so that for any $r\in \bQ_{>0}$,
\[
\|G_r\circ f\|' \le C_1 \deg(G_r\circ f).
\]
Using the definition of $\gamma_r = \cl_{X\times X}(G_r)$ and combining the last two displayed equations, we obtain that
\begin{align*}
r^{i} \, |\Tr(f^*|_{H^{i}(X)})| \le C_1 \|\Delta_{2n-i}\| \deg(G_r \circ f) = C_1 \|\Delta_{2n-i}\| \sum_{j=0}^{n} \binom{n}{j} \deg_j(G_r \circ f).
\end{align*}
Note that
\begin{align*}
\deg_j(G_r \circ f) &= (G_r \circ f)^*H_X^j\cdot H_X^{n-j} = (f^*G_r^*H_X^j)\cdot H_X^{n-j} \\
&= r^{2j} f^*H_X^j\cdot H_X^{n-j} = r^{2j} \deg_j(f).
\end{align*}
It thus follows that
\begin{equation*}
r^{i} \, |\Tr(f^*|_{H^{i}(X)})| \le C_2 \max_{0\le j\le n} r^{2j}\deg_j(f),
\end{equation*}
where the constant $C_2 \coloneqq (n+1) 2^{n} C_1 \max_i \|\Delta_{i}\| > 0$ is independent of $r\in \bQ_{>0}$ and $f$.
Note that by \cref{lemma:log-concave}, the degree sequence $\{\deg_j(f)\}_j$ of the irreducible correspondence $f$ is log-concave.
Hence \cref{lemma:Gr+log-concave} follows directly from \cref{lemma:log-concave-sequence} with $b_i\coloneqq |\Tr(f^*|_{H^{i}(X)})|$ and $a_j\coloneqq C_2\deg_j(f)$.
\end{proof}

\subsection{Proof of Theorem \ref{thm:B}}

\begin{proof}[Proof of \cref{thm:B}\eqref{Assertion:B-1}]
Let $f$ be an arbitrary dynamical correspondence of $X$.
Write $f = \sum_i a_i f_i$, where $a_i\in \bQ_{>0}$ and the $f_i$ are irreducible dynamical correspondences of $X$.
It is thus straightforward to see that
\[
|\Tr(f^*|_{H^{2k}(X)})| \le \sum_i a_i \, |\Tr(f_i^*|_{H^{2k}(X)})| \le \sum_i C \deg_k(a_i f_i) = C \deg_k(f),
\]
where the second inequality follows from \cref{lemma:Gr+log-concave} applied to the irreducible $f_i$.
We hence obtain the inequality \eqref{eq:tr-deg-even}.
As for odd indices, we note that by the Cauchy--Schwarz inequality, the function $\sqrt{\deg_k(\cdot)\deg_{k+1}(\cdot)}$ defined over effective correspondences of $X$ is superadditive.
Precisely, for any two effective correspondences $g$, $h$ of $X$, we always have
\[
\sqrt{\deg_k(g)\deg_{k+1}(g)} + \sqrt{\deg_k(h)\deg_{k+1}(h)} \le \sqrt{\deg_k(g+h)\deg_{k+1}(g+h)}.
\]
In particular, one has
\[
\sum_i \sqrt{\deg_k(a_i f_i)\deg_{k+1}(a_i f_i)} \le \sqrt{\deg_k(f)\deg_{k+1}(f)}.
\]
So similarly, applying \cref{lemma:Gr+log-concave} for each $f_i$, we have
\[
|\Tr(f^*|_{H^{2k+1}(X)})| \le \sum_i a_i \, |\Tr(f_i^*|_{H^{2k+1}(X)})| \le C \sqrt{\deg_k(f)\deg_{k+1}(f)}.
\]
The inequality \eqref{eq:tr-deg-odd} follows and hence we prove \cref{thm:B}\eqref{Assertion:B-1}.
\end{proof}

\begin{proof}[Proof of \cref{thm:B}\eqref{Assertion:B-2}]
Let $f$ be a dynamical correspondence of $X$ such that the $m$-th iterate $f^{\circ m}$ is still a dynamical correspondence of $X$ for any $m\in \bN_{>0}$.
Then it follows from \cref{rmk:norm-eff-corr} that
\[
\deg_k(f^{\circ m}) = \big\|(f^{\circ m})^*|_{\Nk}\big\|_1.
\]
By applying the inequality \eqref{eq:tr-deg-even} in \cref{thm:B}\eqref{Assertion:B-1} to $f^{\circ m}$ for all $m$, there is a uniform positive constant $C>0$, independent of $f^{\circ m}$ or $m$, such that 
\[
|\Tr((f^{\circ m})^*|_{H^{2k}(X)})| \le C \, \big\|(f^{\circ m})^*|_{\Nk}\big\|_1.
\]
Then by the functoriality (see \cref{prop:functorial}), the above inequality becomes
\[
|\Tr((f^*)^m|_{H^{2k}(X)})| \le C \, \big\|(f^{*})^m|_{\Nk}\big\|_1.
\]
Taking limits of the $m$-th roots of both sides yields that $\rho(f^*|_{H^{2k}(X)}) \le \rho(f^*|_{\Nk})$.
Note that the converse direction has been proved in \cref{lemma:easy-direction} so that we actually have the equality \eqref{eq:rho-even}.

Similarly, one can deduce the inequality \eqref{eq:rho-odd} from the inequality \eqref{eq:tr-deg-odd} in \cref{thm:B}\eqref{Assertion:B-1}.
The last assertion that \cref{conj:Serre} holds on $X$ follows readily from the inequalities \eqref{eq:tr-deg-even} and \eqref{eq:tr-deg-odd} by the proof of \cref{lemma:DDC-Serre}.
We hence conclude the proof of \cref{thm:B}.
\end{proof}

\subsection{Proof of Theorem \ref{thm:C}}

\begin{proof}[Proof of \cref{thm:C}\eqref{Assertion:C-1}]
%Let $f$ be an effective correspondence of $X$.
Let $f$ be a dynamical correspondence of $X$.
Then by \cref{rmk:norm-eff-corr}, we see that $\deg_k(f) \sim \|{f^*|_{\Nk}}\|$ for any $0\le k\le n$.
As in the proof of \cref{thm:B}\eqref{Assertion:B-1}, it suffices to consider the case when $f$ is irreducible.
In this case, the degree sequence $\{\deg_j(f)\}_j$ is log-concave by \cref{lemma:log-concave}.
Since Conjecture~\hyperref[conj:Gr]{$G_r(X)$} holds, we see that for any $r\in \bQ_{>0}$, there is a correspondence $G_r$ of $X$ such that $\gamma_r = \cl_{X\times X}(G_r)$, and a constant $C_1>0$, independent of $f$ or $r$, such that
\begin{equation}
\label{eq:proof-C-1}
\|G_r\circ f\| \le C_1 \deg(G_r\circ f).
\end{equation}
For simplicity, we denote by $g_r$ the composite correspondence $G_r \circ f$.
By the definition of $G_r$ or $\gamma_r$, we immediately have for any $0\le i\le 2n$,
\begin{equation}
\label{eq:proof-C-2}
\big\|g_r^*|_{H^{i}(X)}\big\|_{\iota} = r^{i} \big\|f^*|_{H^{i}(X)}\big\|_{\iota}.
\end{equation}

On the other hand, by assumption, Conjecture~\hyperref[conj:Dk]{$D^n(X\times X)$} holds.
This allows us to identify $\N^n(X\times X)_{\bQ_\ell}$ with the $\bQ_\ell$-vector subspace $H_{\alg}^{2n}(X\times X)$ of $H^{2n}(X\times X)$ (see \cite[Proposition~3.6]{Kleiman68}).
In particular, it is convenient to consider the following inclusions of vector spaces:
\[
\begin{tikzcd}
\N^n(X\times X)_{\bQ} \arrow[r, hook] \arrow[d, hook] & \N^n(X\times X)_{\bQ_\ell} \arrow[d, hook] \arrow[r, hook, "\cl_{X\times X}"] & H^{2n}(X\times X) \arrow[d, hook] \\
\N^n(X\times X)_{\bR} \arrow[r, hook] & \N^n(X\times X)_{\bC} \arrow[r, hook] & H^{2n}(X\times X) \otimes_{\bQ_\ell} \bC.
\end{tikzcd}
\]
Note that by definition the norm $\norm{\cdot}_{\iota}$ on $H^{2n}(X\times X)$ is induced from a norm on the complex vector space $H^{2n}(X\times X) \otimes_{\bQ_\ell} \bC$ with respect to a filed isomorphism $\iota\colon \ol\bQ_\ell \isom \bC$.
Without loss of generality, we may adopt an induced norm on $\N^n(X\times X)_{\bR}$ from the same one on $H^{2n}(X\times X) \otimes_{\bQ_\ell} \bC$.
It thus follows that the norm $\|\cl_{X\times X}(g_r)\|_{\iota}$ of $\cl_{X\times X}(g_r) \in H^{2n}(X\times X)$ is equivalent to the norm $\|g_r\|$ of $g_r$ as an element in $\N^n(X\times X)_{\bR}$.
Furthermore, by Poincar\'e duality and the K\"unneth formula, we know that
\[
H^{2n}(X\times X) = \bigoplus_{i=0}^{2n} H^{i}(X)\otimes H^{2n-i}(X) \isom \bigoplus_{i=0}^{2n} \End_{\bQ_\ell}(H^{i}(X)).
\]
Therefore, for any $0\le i\le 2n$, we have
\begin{equation}
\label{eq:proof-C-3}
\big\|g_r^*|_{H^{i}(X)}\big\|_{\iota} \le \Big\|\bigoplus_{i=0}^{2n} g_r^*|_{H^{i}(X)}\Big\|_{\iota} = \big\|\cl_{X\times X}(g_r)\big\|_{\iota} = \|g_r\|.
\end{equation}

Putting the above \cref{eq:proof-C-1,eq:proof-C-2,eq:proof-C-3} together, we obtain that for any $r\in \bQ_{>0}$,
\begin{equation*}
r^{i} \big\|f^*|_{H^{i}(X)}\big\|_{\iota} \le C_1 \deg(G_r\circ f) \le C_2 \, \max_{0\le j\le n} r^{2j} \deg_j(f),
\end{equation*}
where $C_2 \coloneqq (n+1) 2^{n} C_1$ is a constant independent of $f$ and $r$.
Then the inequalities \eqref{eq:Dinh-even} and \eqref{eq:Dinh-odd} in \cref{conj:NC} follow from \cref{lemma:log-concave-sequence} with $b_i\coloneqq \|f^*|_{H^{i}(X)}\|_{\iota}$ and $a_j\coloneqq C_2\deg_j(f)$.
We thus prove \cref{thm:C}\eqref{Assertion:C-1}.
\end{proof}

\begin{proof}[Proof of \cref{thm:C}\eqref{Assertion:C-2}]
It is contained in \cref{lemma:norm-DDC}.
\end{proof}

\begin{proof}[Proof of \cref{thm:C}\eqref{Assertion:C-3}]
It follows from \cref{lemma:DDC-Serre}.
\end{proof}

\begin{proof}[Proof of \cref{thm:C}\eqref{Assertion:C-4}]
It follows from \cref{lemma:norm-SS}.
\end{proof}

\subsection{Proofs of Theorems \ref{thm:A} and \ref{thm:D}}

We first show that abelian varieties satisfy \cref{conj:Gr}. The multiplication map on abelian varieties plays an important role in the proof.

\begin{lemma}
\label{lemma:Gr-holds-AV}
Let $X$ be an abelian variety of dimension $n$ over $\bk$.
Then for any $r\in \bQ_{>0}$, there is a correspondence $G_r$ of $X$ such that $\gamma_{r} = \cl_{X\times X}(G_r)$;
moreover, for any effective correspondence $f$ of $X$, the composition $G_{r}\circ f$ is also effective.

In particular, \cref{conj:Gr} holds on abelian varieties (for general effective correspondences not only dynamical correspondences).
\end{lemma}
\begin{proof}
Let $r\in \bQ_{>0}$ be arbitrary and write $r=n_1/n_2$, where $n_1, n_2\in \bN_{>0}$ are two coprime positive integers.
Let $\Gamma_{[n_i]}$ denote the graph of the multiplication-by-$n_i$ map $[n_i]$ of $X$.
Set
\[
G_r \coloneqq \frac{1}{n_2^{2n}} \, \Gamma_{[n_1]} \circ \Gamma_{[n_2]}^\sT.
\]
Then it is easy to check that $\gamma_r = \cl_{X\times X}(G_r)$, i.e., $G_r^*\alpha = r^i \alpha$ for any $\alpha\in H^i(X)$.

Let $f$ be an arbitrary effective correspondence of $X$.
Then we have
\begin{align*}
G_r \circ f &= \frac{1}{n_2^{2n}} \, \Gamma_{[n_1]} \circ \Gamma_{[n_2]}^\sT \circ f \\
&= \frac{1}{n_2^{2n}} \, \Gamma_{[n_1]} \circ f \circ \Gamma_{[n_2]}^\sT \\
&= \frac{1}{n_2^{2n}} \, ([n_2]\times [n_1])_*(f),
\end{align*}
where the second equality holds since the pullback $f^*$ on cohomology $H^{\bullet}(X)$ is a $\bQ_\ell$-linear map and the third equality follows from Lieberman's \cref{lemma:Lieberman}.
Since effective cycles are preserved by proper pushforward of morphisms, $G_r\circ f$ is effective.
It thus follows that using the norm given by \cref{eq:corr-norm}, we have
\[
\| G_r\circ f \| = \deg(G_r \circ f).
\]
Since all norms on $\N^n(X\times X)_\bR$ are equivalent, \cref{conj:Gr} holds on $X$.
\end{proof}

\begin{remark}
\begin{enumerate}[(1)]
\item One can see from the proof that if $r$ is a positive integer (i.e., $n_2=1$), then $G_r$ is nothing but the graph $\Gamma_{[r]}$ of the multiplication-by-$r$ map $[r]$ on $X$.
Further, for rational $r$, one can check using Lieberman's \cref{lemma:Lieberman} that
\[
G_r = \frac{1}{n_2^{2n}} \, ([n_2]\times [n_1])_*(\Delta_X) = \frac{1}{n_2^{2n}} \, \{(n_2 x, n_1 x) \in X\times X : x\in X\},
\]
which plays a role like the "multiplication-by-$r$ map" on $X$.

\item Note that in the proof of \cref{lemma:Gr-holds-AV}, we actually show that $G_r\circ f$ is effective for every effective correspondence $f$ of $X$.
Consider the $G_r$ action on the correspondences of a general variety $X$ by the left composition.
It is interesting to know whether $G_r$ (if exists) preserves the effective cone $\Eff^n(X\times X)$ in $\N^n(X\times X)_\bR$.
If these were true, then \cref{conj:Gr} holds immediately.
\end{enumerate}
\end{remark}

\begin{proof}[Proof of \cref{thm:A}\eqref{Assertion:A-1}: the abelian variety case]
Let $f$ be an algebraically stable dynamical correspondence of an abelian variety $X$.
Then $f^{\circ m} \sim_{\hom} f^{\Diamond m}$ is still a dynamical correspondence of $X$ for any $m\in \bN_{>0}$ (see \cref{def:alg-stable}).
In this case, all dynamical degrees $\chi_i(f)_{\iota}$, $\lambda_k(f)$ are nothing but the corresponding spectral radii.
Hence the first half of \cref{thm:A}\eqref{Assertion:A-1} follows immediately from \cref{lemma:Gr-holds-AV} and \cref{thm:B}\eqref{Assertion:B-2}.
Then the second half is then a consequence of the first half (see also the proof of \cref{lemma:DDC-Serre}).
\end{proof}

To prove the Kummer surface case of \cref{thm:A}\eqref{Assertion:A-1}, by the same argument as above it suffices to show that \cref{conj:Gr} holds on Kummer surfaces.

\begin{lemma}
\label{lemma:Gr-Kummer}
Let $S$ be a Kummer surface associated to an abelian surface $A$ over $\bk$. 
Then for any $r\in \bQ_{>0}$, there is a correspondence $G_r$ of $S$ such that $\gamma_{r} = \cl_{S\times S}(G_r)$;
moreover, for any dynamical correspondence $f$ of $S$, there is a constant $C>0$, independent of $r$ and $f$, such that
\begin{equation}
\label{eq:Gr-surface}
\|G_r\circ f\| \le C \deg(G_r\circ f).
\end{equation}
\end{lemma}
\begin{proof}
It is a classical result that the standard conjecture \hyperref[conj:C]{$C$} holds on surfaces (see, e.g., \cite[Corollary~2A10]{Kleiman68}).
Hence $G_r$ can be chosen as $\sum_{i=0}^4 r^i \Delta_i$, where the $\Delta_i$ are K\"unneth projectors.
Let $f$ be a dynamical correspondence of $S$.
It remains to show the inequality \eqref{eq:Gr-surface}, which is equivalent to
\[
\|G_r\circ f\| \lesssim \max_{0\le j\le 2} r^{2j} \deg_j(f).
\]
We henceforth use the symbol $\lesssim$ to simplify the narrative.
Choose correspondences $c_t$ of $S$ such that their numerical classes form a basis of $\N^{2}(S\times S)_\bR$.
One can check that $\|f\| \coloneqq \max_t |f\cdot c_t|$ gives a norm on $\N^{2}(S\times S)_\bR$.
Hence it suffices to show that for any $c_t$,
\begin{equation}
\label{eq:Gr-ct}
|(G_r\circ f) \cdot c_t| \lesssim \max_{0\le j\le 2} r^{2j} \deg_j(f).
\end{equation}

Decompose $c_t$ as $\sum_{i=0}^4 c_{t,i}$ with $c_{t,i}\coloneqq \Delta_i \circ c_t$, whose
cohomology class is in $H^{4-i}(S)\otimes H^i(S)$.
Since $G_r = \sum_{i=0}^4 r^i \Delta_i$, it follows from the Lefschetz trace formula (see \cref{prop:Lefschetz-trace-formula}) that
\begin{equation}
\label{eq:Gr-ct-i}
|(G_r\circ f) \cdot c_t| \le \sum_{i=0}^4 |(G_r\circ f) \cdot c_{t,4-i} | = \sum_{i=0}^4 r^i \, |f \cdot c_{t,4-i} |.
\end{equation}

We aim to give an appropriate estimate about $r^i \, |f \cdot c_{t,4-i}|$.
By definition, there exist a birational morphism $\tau\colon \widetilde{A}\to A$ from a smooth projective surface $\widetilde{A}$ to $A$ so that the composite $\pi\colon \widetilde{A}\to S$ is a finite surjective morphism.
Then we denote by $g$ the dynamical pullback $(\pi\times \pi)^{\dynpb}(f)$ of $f$ under $\pi\times \pi$, which is a dynamical correspondence of $\widetilde{A}$ by definition.
We further denote by $h$ the proper pushforward $(\tau\times \tau)_*(g)$ of $g$ under $\tau\times \tau$.
One can check that $h$ is a dynamical correspondence of $A$.
%We thus have the following diagram:
%\[
%\xymatrix{
%A \ar@{|-}[d]_{h} & \widetilde{A} \ar@{->}[l]_{\tau} \ar@{->}[r]^{\pi} \ar@{|-}[d]_{g} & S \ar@{|-}[d]^{f} \\
%A & \widetilde{A} \ar@{->}[l]_{\tau} \ar@{->}[r]^{\pi} & S. }
%\]
\begin{claim}
\label{claim:e}
$e \coloneqq (\tau\times \tau)^{*}(h) - g$ is an effective correspondence of $\widetilde{A}$, whose cohomology class $\cl_{\widetilde{A} \times \widetilde{A}}(e)$ belongs to $H^2(\widetilde{A})\otimes H^2(\widetilde{A})$.
\end{claim}
\begin{proof}[Proof of \cref{claim:e}]\renewcommand{\qedsymbol}{}
Note that the dynamical pullback $(\tau\times \tau)^{\dynpb}(h)$ coincides with the strict transform $(\tau\times \tau)^{-1}_*(h)$ of $h$ under $(\tau\times \tau)^{-1}$ by definition and hence both are equal to $g$ since $g$ is a dynamical correspondence.
Let $(\tau\times \tau)^{-1}(h)$ be the total inverse image of $h$ under the birational morphism $\tau\times \tau$.
Then the difference $(\tau\times \tau)^{-1}(h) - (\tau\times \tau)^{\dynpb}(h)$ is supported inside the union of exceptional divisors $E\times \widetilde{A}$ or $\widetilde{A}\times E$ for some $\tau$-exceptional curves $E$.
As $h$ is a dynamical correspondence of $A$, any irreducible component of the difference cannot be an exceptional divisor of $\tau\times \tau$, so has dimension $2$.
Therefore, the pullback $(\tau\times \tau)^{*}(h)$ of $h$ under $\tau\times \tau$ is supported on $(\tau\times \tau)^{-1}(h)$ and the difference
$e\coloneqq (\tau\times \tau)^{*}(h) - g$
is an effective correspondence of $\widetilde{A}$.
The second half of \cref{claim:e} follows from $\Supp(e) \subsetneq \cup_E ((E\times \widetilde{A}) \cup (\widetilde{A} \times E))$, where $E$ runs over all $\tau$-exceptional curves.
\end{proof}

Note that by \cref{rmk:dyn-pullback} one has $(\pi\times \pi)_*(g) = d^2 f$ with $d\coloneqq \deg(\pi)$.
It follows that 
\begin{align*}
d^2 \, |f\cdot c_{t,4-i}| &= |(\pi\times \pi)_*(g)\cdot c_{t,4-i}| = |g\cdot (\pi\times \pi)^*c_{t,4-i}| \\
&= |((\tau\times \tau)^{*}(h) - e)\cdot (\pi\times \pi)^*c_{t,4-i}| \\
&\le |(\tau\times \tau)^{*}(h)\cdot (\pi\times \pi)^*c_{t,4-i}| + |e\cdot (\pi\times \pi)^*c_{t,4-i}| \\
&= |h\cdot (\tau\times \tau)_*(\pi\times \pi)^*c_{t,4-i}| + |e\cdot (\pi\times \pi)^*c_{t,4-i}|.
\end{align*}
It remains to estimate the above two summands.
First we note that $h$ is a dynamical correspondence of the abelian surface $A$.
Then we have
\begin{align*}
r^i \, |h\cdot (\tau\times \tau)_*(\pi\times \pi)^*c_{t,4-i}| &= |(G_r\circ h)\cdot (\tau\times \tau)_*(\pi\times \pi)^*c_{t,4-i}| \\
&\le \big\| (\tau\times \tau)_*(\pi\times \pi)^*c_{t,4-i}\big\| \big\|G_r\circ h\big\|' \\
&\lesssim \deg(G_r\circ h) \sim \max_{0\le j\le 2} r^{2j} \deg_j(h) \\
&\sim \max_{0\le j\le 2} r^{2j} \deg_j(f),
\end{align*}
where the first equality follows from the fact $(\tau\times \tau)_*(\pi\times \pi)^*c_{t,4-i} \in H^i(A)\otimes H^{4-i}(A)$;
the first inequality holds by the definition of the dual norm \eqref{eq:corr-dual-norm};
the second inequality is guaranteed by \cref{lemma:Gr-holds-AV};
$\deg_j(h)\sim \deg_j(g) \sim \deg_j(f)$ follows from \cref{lemma:deg-k-invariant-under-gen-finite}.
Secondly, by \cref{claim:e}, $|e\cdot (\pi\times \pi)^*c_{t,4-i}|=0$ unless $i=2$ and by \cref{lemma:boundedness} we have that
\[
|e\cdot (\pi\times \pi)^*c_{t,2}| \lesssim \deg_1(e),
\]
which, thanks to \cref{claim:e} again, is further bounded above by
\begin{align*}
\deg_1((\tau\times \tau)^{*}(h)) &= (\tau\times \tau)^{*}(h) \cdot \pr_1^*H_{\widetilde{A}} \cdot \pr_2^*H_{\widetilde{A}} \\
&= h \cdot (\tau\times \tau)_*(\pr_1^*H_{\widetilde{A}} \cdot \pr_2^*H_{\widetilde{A}}).
\end{align*}
Note that $(\tau\times \tau)_*(\pr_1^*H_{\widetilde{A}} \cdot \pr_2^*H_{\widetilde{A}}) \in \N^1(A)_\bR\otimes \N^1(A)_\bR$.
Hence using the dual operator norm \eqref{eq:corr-operator-norm'}, \cref{lemma:norm}, and \cref{rmk:norm-eff-corr}, one has
\[
h \cdot (\tau\times \tau)_*(\pr_1^*H_{\widetilde{A}} \cdot \pr_2^*H_{\widetilde{A}}) \lesssim \big\|h^*|_{\N^1(A)_\bR}\big\|' \sim \big\|h^*|_{\N^1(A)_\bR}\big\|_1 \sim \deg_1(h) \sim \deg_1(f).
\]

Now, putting the last five displayed equations together, we get that
\[
r^i \, |f\cdot c_{t,4-i}| \lesssim \max_{0\le j\le 2} r^{2j} \deg_j(f).
\]
Combining with \cref{eq:Gr-ct-i} yields \cref{eq:Gr-ct}.
This concludes the proof of \cref{lemma:Gr-Kummer}.
\end{proof}

\begin{proof}[Proof of \cref{thm:A}\eqref{Assertion:A-1}: the Kummer surface case]
Similar with the abelian variety case, it follows from \cref{lemma:Gr-Kummer} and \cref{thm:B}\eqref{Assertion:B-2}.
\end{proof}

The following might be well known to experts but we are not able to find a precise reference in the literature.

\begin{lemma}
\label{lemma:D-Kummer2}
Let $S$ be a Kummer surface associated to an abelian surface $A$ over $\ol\bF_p$.
Then Conjecture~\hyperref[conj:Dk]{$D^2(S\times S)$} holds.
\end{lemma}
\begin{proof}
Let $\tau \colon \wtilde A \to A $ be the blowup of $A$ at the sixteen $2$-torsion points so that $\pi\colon \wtilde A \to S$ is a finite surjective morphism.
It suffices to show that Conjecture~\hyperref[conj:Dk]{$D^2(\wtilde A\times \wtilde A)$} holds.
Indeed, suppose, for the time being, that this has been proved.
Let $Z_{S\times S}$ be a numerically trivial $2$-cycle on $S\times S$.
Since $(\pi\times \pi)_*(\pi\times \pi)^* = \deg(\pi)^2$, we see that $(\pi\times \pi)^*$ is injective and $(\pi\times \pi)_*$ is surjective.
Thus, the pullback $(\pi\times \pi)^*Z_{S\times S}$ on $\wtilde A\times \wtilde A$ is numerically trivial and hence homologically trivial by assumption.
Now, let $\alpha_{S\times S}\in H^4(S\times S)$ be arbitrary.
Choose $\alpha_{\wtilde A\times \wtilde A}\in H^4(\wtilde A\times \wtilde A)$ such that $(\pi\times \pi)_*\alpha_{\wtilde A\times \wtilde A} = \alpha_{S\times S}$.
It follows that
\begin{align*}
\cl_{S\times S}(Z_{S\times S}) \cup \alpha_{S\times S} &= \cl_{S\times S}(Z_{S\times S}) \cup (\pi\times \pi)_*\alpha_{\wtilde A\times \wtilde A} \\
&= \cl_{\wtilde A\times \wtilde A}((\pi\times \pi)^*Z_{S\times S}) \cup \alpha_{\wtilde A\times \wtilde A} \\
&= 0.
\end{align*}
We thus show that $Z_{S\times S}$ is homologically trivial and hence Conjecture~\hyperref[conj:Dk]{$D^2(S\times S)$} holds.

We now aim to prove that Conjecture~\hyperref[conj:Dk]{$D^2(\wtilde A\times \wtilde A)$} holds.
%Note first that $\wtilde A\times \wtilde A$ has the Chow--K\"unneth decomposition (see \cite[\S 6.1.5]{MNP13}):
%\[
%\ch(\wtilde A\times \wtilde A) = \bigoplus_{i=0}^8 \ch^i(\wtilde A\times \wtilde A) = \bigoplus_{i=0}^8 \bigoplus_{r+s=i} \ch^r(\wtilde A)\otimes \ch^r(\wtilde A).
%\]
%Also, by the blowup formula (see \cite[Examples~2.8.1]{MNP13}), one has $\ch(\wtilde A) = \ch(A) \oplus \bbL^{\oplus 16}$.
%Combining them together, one can show that
%\begin{align*}
%\ch^4(\wtilde A\times \wtilde A) &= \ch^4(A\times A) \oplus (\ch^2(A) \otimes \bbL^{\oplus 16}) \oplus (\bbL^{\oplus 16} \otimes \ch^2(A)) \oplus (\bbL^{\oplus 16}\otimes \bbL^{\oplus 16}) \\
%\ch^3(\wtilde A\times \wtilde A) &= \ch^3(A\times A) \oplus (\ch^1(A) \otimes \bbL^{\oplus 16}) \oplus (\bbL^{\oplus 16} \otimes \ch^1(A)) \\
%\ch^2(\wtilde A\times \wtilde A) &= \ch^2(A\times A) \oplus (\ch^0(A) \otimes \bbL^{\oplus 16}) \oplus (\bbL^{\oplus 16} \otimes \ch^0(A)).
%\end{align*}
%It thus follows from \cite[Proposition~7.6.1]{MNP13} that
%\begin{align*}
%&\quad \ \CH^2(\wtilde A\times \wtilde A)_\bQ = \bigoplus_{i=2}^4 \CH^2(\ch^i(\wtilde A\times \wtilde A)) \\
%&\isom \CH^2(A\times A)_\bQ \oplus (\CH^1(A)_\bQ \otimes \bQ^{\oplus 16}) \oplus (\bQ^{\oplus 16} \otimes \CH^1(A)_\bQ) \oplus (\bQ^{\oplus 16} \otimes \bQ^{\oplus 16}).
%\end{align*}
First, by the K\"unneth formula, we have that
\[
H^4(\wtilde A \times \wtilde A) \isom H^4(A\times A) \oplus (H^2(A) \otimes \bQ_\ell^{\oplus 16}) \oplus (\bQ_\ell^{\oplus 16} \otimes H^2(A)) \oplus (\bQ_\ell^{\oplus 16} \otimes \bQ_\ell^{\oplus 16}),
\]
where $\bQ_\ell^{\oplus 16}$ represents the $\bQ_\ell$-subspace of $H^2(\wtilde A)$ generated by the sixteen $\tau$-exceptional curves on $\wtilde A$.
Note that the above decomposition is orthogonal with respect to the cup product.
Thanks to \cite[Theorem~3.5]{Zarhin94}, the Tate conjecture $T^2(A\times A)$ holds for the square $A\times A$ of $A$ (see \cite{Tate94} for details on Tate's conjectures).
Note that by Weil's work \cite{Weil48}, the Frobenius endomorphism acts semisimply on \'etale cohomology of abelian varieties.
It thus follows from \cite[Propositions~8.2 and 8.4]{Milne86-AJM} that the standard conjecture \hyperref[conj:Dk]{$D^2(A\times A)$} holds (see also \cite[Proposition~2.6]{Tate94}).
In particular, we have an orthogonal decomposition of $H^4(A\times A)$ as $H_{\alg}^4(A\times A) \oplus H_{\tr}^4(A\times A)$ with respect to the cup product by \cref{lemma:decomposition}.
It is also well known that there is an orthogonal decomposition of $H^2(A)$ as $H_{\alg}^2(A) \oplus H_{\tr}^2(A)$.
Therefore, the algebraic part $H_{\alg}^4(\wtilde A \times \wtilde A)$ of $H^4(\wtilde A \times \wtilde A)$ is equal to
\[
H_{\alg}^4(A\times A) \oplus (H_{\alg}^2(A) \otimes \bQ_\ell^{\oplus 16}) \oplus (\bQ_\ell^{\oplus 16} \otimes H_{\alg}^2(A)) \oplus (\bQ_\ell^{\oplus 16} \otimes \bQ_\ell^{\oplus 16}),
\]
and the transcendental part $H_{\tr}^4(\wtilde A \times \wtilde A)$ of $H^4(\wtilde A \times \wtilde A)$, i.e., the orthogonal complement of $H_{\alg}^4(\wtilde A \times \wtilde A)$ with respect to the cup product, equals
\[
H_{\tr}^4(A\times A) \oplus (H_{\tr}^2(A) \otimes \bQ_\ell^{\oplus 16}) \oplus (\bQ_\ell^{\oplus 16} \otimes H_{\tr}^2(A)).
\]
We thus have an orthogonal decomposition $H^4(\wtilde A \times \wtilde A) = H_{\alg}^4(\wtilde A \times \wtilde A) \oplus H_{\tr}^4(\wtilde A \times \wtilde A)$.
By \cref{lemma:decomposition} again, Conjecture~\hyperref[conj:Dk]{$D^2(\wtilde A\times \wtilde A)$} holds and hence \cref{lemma:D-Kummer2} follows.
\end{proof}

\begin{proof}[Proof of \cref{thm:A}\eqref{Assertion:A-2}]
Let $S$ be a Kummer surface over $\ol\bF_p$.
Then \cref{conj:NC} holds on $S$ by \cref{lemma:Gr-Kummer,lemma:D-Kummer2}, and \cref{thm:C}(1).
The rest follows from \cref{thm:C}(2) and (4).
\end{proof}

\begin{proof}[Proof of \cref{thm:A}\eqref{Assertion:A-3}]
Let $X$ be an abelian variety over $\ol \bF_p$.
Then there is a set of primes of positive density such that the standard conjecture \hyperref[conj:Dk]{$D^n(X\times X)$} holds for $\ell$-adic \'etale cohomology whenever $\ell$ is in this set; see Clozel \cite{Clozel99}.
Pick such a prime $\ell$.
Then \cref{conj:NC} and hence \cref{conj:DDC} hold on $X$ for this specific $\ell$ by \cref{lemma:Gr-holds-AV} and \cref{thm:C}\eqref{Assertion:C-1}.
\end{proof}

\begin{proof}[Proof of \cref{thm:D}]
Let $f$ be an effective finite correspondence of the abelian variety $X$.
Then $f^{\circ m}$ is effective for any $m\in \bN_{>0}$ (see \cref{rmk:finite-corr-comp}).
So, similar to the proof of the abelian variety case of \cref{thm:A}\eqref{Assertion:A-1}, \cref{thm:D} follow from \cref{lemma:Gr-holds-AV} and \cref{thm:B}\eqref{Assertion:B-2} with "dynamical correspondence" replaced by "effective finite correspondence" (see also \cref{rmk:finite-corr-alg-stable}).
\end{proof}

%%%%%%%%%%%%%%%%%%%%%%%%%%%%%%%%%%%%%%%%%%%%%%%%%%%%%%%%%%

%\linespread{1.1}

%\bibliographystyle{amsplain}
\bibliographystyle{amsalpha}
\bibliography{../mybib}

\end{document}